\documentclass[draft,12pt,reqno,a4paper]{amsart}

\usepackage[margin=3cm,footskip=1cm]{geometry}

\usepackage{amssymb} \usepackage{amsmath} \usepackage{amsfonts}
\usepackage{amsthm} \usepackage{mathtools}

\usepackage{enumerate} \usepackage[latin1]{inputenc}
\usepackage[shortlabels]{enumitem}

\usepackage{color} 
\usepackage{graphicx} \usepackage{verbatim}

\DeclareMathOperator*{\slim}{s--lim}

\DeclareMathOperator*{\vGlim}{{\mathcal G}--lim}

\newcommand{\supp}{\operatorname{supp}}

\newcommand{\R}{{\mathbb{R}}}

 \renewcommand{\c}{{\rm c}}
\newcommand{\e}{{\rm e}} \newcommand{\ess}{{\rm ess}}
 \renewcommand{\i}{{\rm i}}
\renewcommand{\d}{{\rm d}}

\DeclarePairedDelimiter\inp\langle\rangle

\newcommand\parb[2][]{#1 \big ( #2#1\big )} \newcommand\parbb[2][]{#1
  \Big ( #2#1\Big )} 
 \newcommand{\grad}{{\rm grad }\,}
 
 \renewcommand{\exp}{{\rm exp}}
 \newcommand{\mfor}{\text{ for }}

\newcommand{\vA}{{\mathcal A}} 
 
\newcommand{\vE}{{\mathcal E}} 
\newcommand{\vG}{{\mathcal G}} \newcommand{\vI}{{\mathcal I}}
 \newcommand{\vH}{{\mathcal H}}
 
 \newcommand{\vN}{{\mathcal N}}

\theoremstyle{plain}
\newtheorem{thm}{Theorem}[section]
\newtheorem{proposition}[thm]{Proposition}
\newtheorem{lemma}[thm]{Lemma} \newtheorem{corollary}[thm]{Corollary}
\theoremstyle{definition} 
\newtheorem{defn}[thm]{Definition} 

 \newtheorem{cond}[thm]{Condition}
 \newtheorem{remark}[thm]{Remark}

 \newtheorem*{remarks*}{Remarks}
\newtheorem*{remark*}{Remark}

\numberwithin{equation}{section}

\title{Time-dependent  scattering
    theory on manifolds}

\thanks{{K.I. is supported by JSPS KAKENHI grant nr. 25800073, 17K05325. 
E.S. is supported by  DFF grant nr.  4181-00042.}}

\author{K. Ito}
\address[K. Ito]{Graduate School of Mathematical Sciences, The University of Tokyo\\
3-8-1 Komaba, Meguro-ku, Tokyo 153-8914, Japan}
\email{ito@ms.u-tokyo.ac.jp}

\author{E. Skibsted} \address[E. Skibsted]{Institut for Matematiske
  Fag \\
  Aarhus Universitet\\ Ny Munkegade 8000 Aarhus C, Denmark}
\email{skibsted@math.au.dk}

\begin{document}
\begin{abstract}
This is the third and the last paper in a series of papers 
on  spectral and scattering theory for the 
Schr\"odinger operator on a manifold possessing an escape function,
for example  a manifold  with asymptotically Euclidean and/or hyperbolic ends. 
Here we discuss the time-dependent scattering theory.
A long-range perturbation is allowed,
and scattering by obstacles, possibly non-smooth {and/or unbounded
in a certain way}, 
is included in the theory.
We also resolve  a conjecture
by Hempel--Post--Weder on cross-ends transmissions 
between two or more ends,
formulated in a time-dependent manner.
\end{abstract}

\maketitle
\tableofcontents

\section{Introduction}\label{sec:intr} 

Let $(M,g)$ be a connected Riemannian manifold.
In this paper we study time-dependent scattering  theory for the geometric Schr\"odinger operator
\begin{align*}
H=H_0+V; \quad
H_0=-\tfrac12\Delta=\tfrac12p_i^*g^{ij}p_j,\ p_i=-\mathrm i\partial_i,
\end{align*}
on the Hilbert space ${\mathcal H}=L^2(M)$. The potential $V$ is
real-valued and bounded, and the self-adjointness of $H$ is realized
by the Dirichlet boundary condition. 
For previous time-dependent short-range scattering theories on manifolds we
refer to \cite{Co,Il, IN,IS1}. For a review of 
scattering by an   unbounded obstacle in $\R^d$ as studied in \cite{Co,Il}
we refer to \cite{Ya2}.

Our theory includes a generalization of
asymptotic completeness on the  Euclidean space 
stated in
terms of  a ``time-dilated'' comparison dynamics by Yafaev  
\cite{Ya1}.  In its {simplest} form given for  $H=-\tfrac12 \Delta$ on the
Euclidean space $\R^d$ it
amounts to the comparison (using spherical coordinates {$(r,\sigma)$})
\begin{align*}
  \parb{\e^{\pm \i tH}f}{(r,\sigma)}\approx r^{(1-d)/2}\e^{\pm
      \i r^2/(2t)}t^{-1/2}\phi_\pm( r/t,\sigma) \text{ for }t\to \infty,
\end{align*} and showing a unitary correspondance
$f\leftrightarrow \phi_\pm$. 
The virtue of this  type of comparison dynamics is its
simplicity and, more importantly,  that it can be generalized to
manifolds  without  invoking 
 micro-local analysis (for example the Fourier transform). 
It was used
in \cite{IS1} to obtain the asymptotic completeness 
on a manifold with asymptotically Euclidean ends of
short-range type by a time-dependent method. In this paper 
we develop a time-dependent theory applicable to 
both asymptotically Euclidean and hyperbolic ends of long-range type, 
and moreover we show
a  relationship between the time-dependent and stationary wave operators,
the latter of which been  constructed in \cite{IS3}. 
More precisely, let $F^{\pm}$ be the future/past stationary wave operators of
\cite{IS3}  and $W^\pm$ be the time-dependent operators of this paper.
Then we show that 
\begin{align}\label{eq:g}
  F^\pm=(W^\pm)^*.
\end{align} These are relationships for  general long-range models.
However under additional conditions more specialized wave operators of  
`Dollard type'  are constructed by a time-dilated
comparison dynamics modeled after Dollard's construction {on} the
Euclidean space \cite{Do}. Deducing  from  \eqref{eq:g} we then obtain 
similar results for the simplified Dollard type operators. Moreover
yet similar  formulas are  obtained under short-range 
conditions. In particular the latter  formulas complement the purely
time-dependent theory of \cite{IS1}. We use the
Cook--Kuroda method to show the existence of the time-dependent wave
operators in the most general setting. Asymptotic
completeness is then a consequence of \eqref{eq:g} (and its simplified
versions), see \cite{II}
for a  similar application of stationary scattering theory in
the Euclidean context. 

Another result  of this paper is a resolution of a conjecture
of \cite{HPW} on cross-ends transmissions on a manifold with two or
more ends. We give a time-dependent formulation in a strong form, see
Corollary \ref{cor:transmission2}. Again the result is a consequence
of the stationary theory \cite{IS3}. In fact we believe that
stationary theory is essential not   only for our proof but for  
\emph{any} conceivable proof. This is in  contrast to asymptotic
completeness of time-dependent wave operators. 
Although  the
stationary theory is important for our procedure of proof,  
the highly `symmetric' arguments of 
\cite{IS1} would plausibly extend in a modified form to provide a purely time-dependent
proof of  asymptotic
completeness. However this is not pursued by us partly
because in general such procedure  seems to require other  conditions
(and whence would not yield a strict improvement) compared to
the approach given in this paper.

\subsection{Setting and review}\label{subsec:preliminary result}
Our  paper is a direct
continuation of  \cite{IS2,IS3}, and we start by recalling the setting and 
various results from there.
This section exhibits only a minimal review,
and we refer to  \cite[Subsection~1.2]{IS2} for 
several examples of manifolds satisfying the abstract conditions appearing below.

\subsection{Basic framework}

We assume an \textit{end} structure on $M$ in a somewhat disguised form.\begin{cond}\label{cond:12.6.2.21.13}
Let $(M,g)$ be a connected Riemannian manifold  of dimension $d\ge 1$.
There exist a function $r\in C^\infty(M)$ with image $r(M)=[1,\infty)$
and constants 
$c>0$ and $r_0\ge 2$ such that:
\begin{enumerate}
\item\label{item:12.4.7.19.40} 
The gradient vector field $\omega=\mathop{\mathrm{grad}} r\in\mathfrak X(M)$ is
forward complete in the sense that the forward integral curve
$(x(t))_{t\geq 0}$ of $\omega$ is defined
for any initial point $x=x(0)\in M$.
\item\label{item:12.4.7.19.40b}
The bound $|\mathrm dr|=|\omega|\ge c$ holds on $\{x\in M\,|\, r(x)> r_0/2\}$. 
\end{enumerate}
\end{cond}

Under Condition~\ref{cond:12.6.2.21.13}
each component of the subset 
$E=\{x\in M\,|\, r(x)>r_0\}$
is called an \emph{end} of $M$,
and, along with Condition~\ref{cond:12.6.2.21.13a} below, 
the function $r$ may model a distance function there.
We note that by 
Condition~\ref{cond:12.6.2.21.13} (\ref{item:12.4.7.19.40b})
and the implicit function theorem  
the $r$-spheres 
\begin{align*}
S_R=\{x\in M\,|\,r(x)=R\};\quad R> r_0/2,
\end{align*} 
are submanifolds of $M$.
We will later in this  section introduce corresponding  \textit{spherical coordinates} on $E$.

Let us impose more conditions on the geometry of $E$ in terms of the radius function $r$.
Choose $\chi\in C^\infty(\mathbb{R})$ such that 
\begin{align}
\chi(t)
=\left\{\begin{array}{ll}
1 &\mbox{ for } t \le 1, \\
0 &\mbox{ for } t \ge 2,
\end{array}
\right.
\quad
\chi\ge 0,
\quad
\chi'\le 0,\quad \sqrt{1-\chi}\in C^\infty,
\label{eq:14.1.7.23.24}
\end{align}
and set 
\begin{align}
\eta=1-\chi(2r/r_0),\quad
\tilde\eta=|\mathrm dr|^{-2}\eta 
=|\mathrm dr|^{-2}\bigl(1-\chi(2r/r_0)\bigr).
\label{eq:13.9.23.5.54b}
\end{align}
We introduce a  ``radial'' differential operator $A$:
\begin{align}
A
=\mathop{\mathrm{Re}} p^r
=\tfrac12\bigl( p^r+( p^r)^*\bigr);\quad 
p^r=-\mathrm i \nabla^r,\ 
\nabla^r=\nabla_{\omega}=g^{ij}(\nabla_i r)\nabla_j,
\label{eq:13.9.23.2.24}
\end{align}
and also the ``spherical'' tensor $\ell$ 
and the associated differential operator $L$: 
\begin{align}
\ell=g-\tilde\eta\,\mathrm dr\otimes \mathrm dr,\quad 
L=p_i^*\ell^{ij}p_j
.
\label{eq:13.9.23.5.54}
\end{align}
In the spherical coordinates, the tensor $\ell$
may be identified with the pull-back of $g$ to the $r$-spheres. We
call $L$ the spherical part of $-\Delta$. If $|\d r|=1$ then
$-L$ acts as  the Laplace--Beltrami operator on $S_r$ (in general as a kind of
perturbation of this operator).
 We remark that the tensor $\ell$ clearly
satisfies
\begin{align}
0\le \ell\le g,\quad
\ell^{\bullet i}(\nabla r)_i=(1-\eta)\mathrm dr,
\label{eq:14.12.23.13.49}
\end{align}
where the first bounds of \eqref{eq:14.12.23.13.49} are
understood as quadratic form estimates on the fibers
of the tangent bundle of $M$.

Let us recall a local expression of the Levi--Civita connection $\nabla$:
If we denote the Christoffel symbol by 
$\Gamma^k_{ij}=\tfrac12g^{kl}(\partial_i g_{lj}+\partial_j g_{li}-\partial_lg_{ij})$,
then for any smooth function $f$ on $M$
\begin{align}
(\nabla f)_i&=(\nabla_if)=(\mathrm d f)_i=\partial_if,\quad
(\nabla^2 f)_{ij}
=\partial_i\partial_j f-\Gamma^k_{ij}\partial_kf.
\label{eq:14.3.5.0.14}
\end{align}
Note that $\nabla^2f $ is the geometric Hessian of $f$. 

\begin{cond}\label{cond:12.6.2.21.13a}
There exist constants $\tau,C>0$ such that 
globally on $M$
\begin{subequations}{
\begin{align}
\begin{split}
|\nabla|\mathrm dr|^2|&\leq Cr^{-1-\tau/2},\qquad\qquad
\bigl|\ell^{\bullet i}\nabla_i\nabla^r|\mathrm dr|^2\bigr|\le Cr^{-1-\tau/2},\\
|\nabla^k r|&\leq C\text{ for }\,k\in\{1,2\},\qquad\quad 
\bigl|\ell^{\bullet i}\nabla_i\Delta r\bigr|\leq Cr^{-1-\tau/2}.
\end{split}\label{eq:13.9.28.13.28}
\end{align}
In addition, there} exists $\sigma'>0$ such that for all   $R>r_0/2$, and as  quadratic forms on fibers of the tangent bundle of $S_R$, 
\begin{align}
R\,\iota^*_R\nabla^2r\ge \tfrac12\sigma' |\mathrm dr|^2\iota^*_R g,
\label{eq:13.9.5.3.30}
\end{align} where $\iota_R\colon S_R\hookrightarrow M$ 
is the inclusion map.
\end{subequations}
\end{cond}

The second bound of \eqref{eq:13.9.28.13.28} is not necessary for the results of 
\cite{IS2}, but it is for \cite{IS3} and this paper.
We note that 
Condition~\ref{cond:12.6.2.21.13a} and the identity
\begin{align}
(\nabla^2r)^{ij}(\nabla r)_j=\tfrac12(\nabla|\mathrm dr|^2)^i
\label{eq:14.12.27.3.22}
\end{align} was used in \cite{IS2} to obtain the  more practical version
of \eqref{eq:13.9.5.3.30}: For any $\sigma\in(0,\sigma')$ and $\tau$
as in Condition~\ref{cond:12.6.2.21.13a}  there exists $C>0$
such that globally on $M$
\begin{align}
r\parb{\nabla^2r-\tfrac 12\tilde \eta^2(\nabla^r|\mathrm dr|^2)\d r
  \otimes\d r} \ge \tfrac12\sigma |\mathrm dr|^2 \ell- Cr^{-\tau}g.
\label{eq:13.9.5.3.30C}
\end{align} 
 If $|\mathrm dr|=1$  for
    $r>r_0/2$ then \eqref{eq:13.9.5.3.30C} is fulfilled with
    $\sigma=\sigma'$ and any $\tau$.

Next we introduce an effective potential:
\begin{align}
q=V+\tfrac18\tilde\eta \bigl[(\Delta r)^2+2\nabla^r\Delta r\bigr].
\label{eq:14.12.10.22.48}
\end{align}
{Here we remark that 
\begin{align}
\begin{split}
H=\tfrac 12A\tilde\eta A+\tfrac 12 L
+q+\tfrac14(\nabla^r\tilde\eta)(\Delta r)
.
\end{split}\label{eq:15091010b}
\end{align}}

\begin{cond}\label{cond:10.6.1.16.24}
  There exists a splitting by real-valued functions:
  \begin{align*}
  q=q_1+q_2;\quad q_1\in C^1(M)\cap L^\infty(M),\  q_2\in L^\infty(M),
  \end{align*}
such that for some $\rho',C>0$ the following bounds hold globally on $M$:
\begin{align}\label{eq:60}
\nabla ^rq_1\leq Cr^{-1-\rho'},\quad
|q_2|\le Cr^{-1-\rho'}.
\end{align}
\end{cond}

Now let us explain  the self-adjoint realizations of $H$ and $H_0$.
Since $(M,g)$ can be incomplete, the operators $H$ and $H_0$ are not necessarily 
essentially self-adjoint on $C^\infty_{\mathrm{c}}(M)$.
We realize $H_0$ as a self-adjoint operator by imposing the 
Dirichlet boundary condition, i.e.\ 
$H_0$ is the unique self-adjoint
operator associated with the closure of the quadratic form
\begin{align*}
\langle H_0\rangle_\psi=\langle \psi, -\tfrac12 \Delta \psi\rangle,\quad \psi
\in C^\infty_{\mathrm{c}}(M).
\end{align*}
We denote the form closure and the self-adjoint realization by the same symbol $H_0$. 
Define the associated Sobolev spaces $\mathcal H^s$ by
\begin{align}\mathcal H^s=(H_0+1)^{-s/2}{\mathcal H}, \quad s\in\mathbb{R}.
\label{eq:13.9.5.1.2}
\end{align}
Then $H_0$ may be understood as a
closed quadratic form on $Q(H_0)=\mathcal H^1$.
Equivalently, $H_0$ makes sense also as a bounded operator 
$\mathcal H^1\to\mathcal H^{-1}$, whose action coincides with 
that for distributions.  
By the definition of the Friedrichs extension 
the self-adjoint realization of $H_0$ 
is the restriction of such distributional
$H_0\colon \mathcal H^1\to\mathcal H^{-1}$ to the domain:
\begin{align*} 
{\mathcal D}(H_0)=\{\psi\in\mathcal H^1\,|\, H_0\psi\in {\mathcal H}\}\subseteq {\mathcal H}.
\end{align*}
Since $V$ is bounded and self-adjoint by Conditions~\ref{cond:12.6.2.21.13}--\ref{cond:10.6.1.16.24}, 
we can realize the self-adjoint operator $H=H_0+V$ simply as
\begin{align*}
H=H_0+V,\quad \mathcal D(H)=\mathcal D(H_0).
\end{align*}

In contrast to (\ref{eq:13.9.5.1.2}) we introduce the Hilbert spaces $\mathcal H_s$
and $\mathcal H_{s\pm}$ with configuration weights:
\begin{align*}
\mathcal H_s=r^{-s}{\mathcal H}, \quad 
{\mathcal H_{s+}=\bigcup_{s'>s}\mathcal H_{s'},\quad
\mathcal H_{s-}=\bigcap_{s'<s}\mathcal H_{s'},\quad}
s\in\mathbb{R}.
\end{align*}
We consider the $r$-balls $B_R=\{ r(x)<R\}$ and  
the characteristic functions 
\begin{align}
\begin{split}
F_\nu=F(B_{R_{\nu+1}}\setminus B_{R_\nu}),\ R_\nu=2^{\nu},\ \nu\ge 0,
\end{split}\label{eq:13.9.12.17.34}
\end{align}
 where $F(\Omega)=1_\Omega$ is used for the  characteristic
 function of a subset $\Omega\subseteq M$. 
Define the associated Besov spaces $B$ and $B^*$ by
\begin{align}
\begin{split}
B&=\{\psi\in L^2_{\mathrm{loc}}(M)\,|\,\|\psi\|_{B}<\infty\},\quad 
\|\psi\|_{B}=\sum_{\nu=0}^\infty R_\nu^{1/2}
\|F_\nu\psi\|_{{\mathcal H}},\\
B^*&=\{\psi\in L^2_{\mathrm{loc}}(M)\,|\, \|\psi\|_{B^*}<\infty\},\quad 
\|\psi\|_{B^*}=\sup_{\nu\ge 0}R_\nu^{-1/2}\|F_\nu\psi\|_{{\mathcal H}},
\end{split}\label{eq:13.9.7.15.11}
\end{align}
respectively.
We also define $B^*_0$ to be the closure of $C^\infty_{\mathrm c}(M)$ in $B^*$.
Recall the nesting:
\begin{align*}
\mathcal H_{1/2+}\subsetneq B\subsetneq \mathcal H_{1/2}
\subsetneq\vH
\subsetneq\mathcal H_{-1/2}\subsetneq B^*_0\subsetneq B^*\subsetneq \mathcal H_{-1/2-}.
\end{align*}

Using the function $\chi\in C^\infty(\mathbb R)$ of \eqref{eq:14.1.7.23.24},
define $\chi_n,\bar\chi_n,\chi_{m,n}\in C^\infty(M)$ for 
 $n>m\ge 0$ by
\begin{align}
\chi_n=\chi(r/R_n),\quad \bar\chi_n=1-\chi_n,\quad
\chi_{m,n}=\bar\chi_m\chi_n.
\label{eq:11.7.11.5.14}
\end{align}  
 Let us introduce an auxiliary space:
\begin{align*}
\mathcal N=\{\psi\in L^2_{\mathrm{loc}}(M)\,|\, \chi_n\psi\in \mathcal
H^1\mbox{ for all }n\ge 0\}.
\end{align*}
This is a  space of  functions that intuitively satisfy the
  Dirichlet boundary condition, although 
possibly with infinite $\mathcal H^1$-norm on $M$.
 Note that under Conditions~\ref{cond:12.6.2.21.13}--\ref{cond:10.6.1.16.24} the manifold $M$
  may be, e.g.\ a half-space in the Euclidean space  (see \cite[Subsection~1.2]{IS2}), and there could be 
a `boundary' even for large $r$, which in our framework appears
`invisible' from inside $M$ (see  Remark
\ref{remark:extended-framewoM} for some elaboration). Recall a similar interpretation of the
space  $\mathcal H^1$.

\subsection{Review of results from \cite{IS2}}
Now we gather and review the main results from \cite{IS2}.

Our first theorem is Rellich's theorem, the absence of $B^*_0$-eigenfunctions 
with eigenvalues above a certain ``critical energy''
$\lambda_0\in\mathbb R$ defined by 
\begin{align}
\lambda_0
=\limsup_{r\to\infty}q_1
=
\lim_{R\to\infty}
\bigl(\sup\{q_1(x)\,|\, r(x)\ge R\}\bigr).
\label{eq:13.9.30.6.8}
\end{align}
For the Euclidean and the hyperbolic spaces and many other examples 
the critical energy $\lambda_0$ can be computed 
explicitly, and the essential
   spectrum is given by $\sigma_\ess(H)= [\lambda_0,\infty).$
  The latter is  usually seen in terms of Weyl sequences, see  \cite{K}.

\begin{thm}\label{thm:13.6.20.0.10}
Suppose Conditions~\ref{cond:12.6.2.21.13}--\ref{cond:10.6.1.16.24},
and let $\lambda>\lambda_0$.
If a function  $\phi\in  L^2_{\mathrm{loc}}(M)$ satisfies that
\begin{enumerate}
\item\label{item:13.7.29.0.27}
$(H-\lambda)\phi=0$ in the distributional sense, 
\item\label{item:13.7.29.0.26}
$\bar\chi_m\phi\in \mathcal N \cap B_0^*$ for all $m\ge 0$ large enough,
\end{enumerate}
then $\phi=0$ in $M$.
\end{thm}

Next we discuss the limiting absorption principle
and the radiation condition
related to the resolvent $R(z)=(H-z)^{-1}.$ 
We state a locally uniform bound for the resolvent as
a map: $B\to B^*$.
For that we need a compactness condition.
\begin{cond}\label{cond:12.6.2.21.13b}
In addition to
Conditions~\ref{cond:12.6.2.21.13}--\ref{cond:10.6.1.16.24}, 
there exists an open subset   $\mathcal I\subseteq (\lambda_0,\infty)$
such that for any $n\ge 0$ and  compact interval $I\subseteq \mathcal I$ 
the mapping 
\begin{align*}
  \chi_nP_H(I)\colon \mathcal H\to\mathcal H
\end{align*}
is  compact, where $P_H(I)$ denotes the spectral projection onto $I$ for $H$. 
\end{cond} 

Due to Rellich's compact embedding theorem,
``boundedness'' of $r$-balls provides a criterion for Condition~\ref{cond:12.6.2.21.13b}:
If each $r$-ball $B_R$, $R\geq 1$, is 
isometric to a bounded subset of a complete manifold, 
Condition~\ref{cond:12.6.2.21.13b} is satisfied for $\mathcal I=(\lambda_0,\infty)$.

We fix any $\sigma\in (0,\sigma')$ and then 
large enough $C>0$ in agreement with \eqref{eq:13.9.5.3.30C}, 
and introduce the positive quadratic form
\begin{align*}
h:=\nabla^2r-\tfrac 12\tilde \eta^2(\nabla^r|\mathrm dr|^2)\d r
  \otimes\d r+2Cr^{-1-\tau}g\ge \tfrac12\sigma r^{-1}|\mathrm dr|^2 \ell+Cr^{-1-\tau}g.
\end{align*} 
For any  subset $I\subseteq \mathcal I$ we denote
\begin{align*}
I_\pm=\{z=\lambda\pm \mathrm i\Gamma\in \mathbb C\,|\,\lambda\in I,\ \Gamma\in (0,1)\},
\end{align*}
respectively. 
We also use the notation 
$\langle T\rangle_\phi=\langle\phi,T\phi\rangle$.

\begin{thm}\label{thm:12.7.2.7.9}
Suppose Condition~\ref{cond:12.6.2.21.13b}
and let $I\subseteq \mathcal I$ be a compact interval.
Then there exists $C>0$ such that 
for any $\phi=R(z)\psi$ with $z\in I_\pm$ and $\psi\in B$
\begin{align}
\|\phi\|_{B^*}+\|p^r\phi\|_{B^*}
+\langle p_i^*h^{ij}p_j\rangle_\phi^{1/2}
+\|H_0\phi\|_{B^*}
\le C\|\psi\|_B.
\label{eq:13.8.22.4.59c}
\end{align}
\end{thm}

In our theory the Besov boundedness \eqref{eq:13.8.22.4.59c} 
does not immediately imply the limiting absorption principle,
and for the latter  we need also radiation condition bounds implied by
minor 
additional regularity conditions.

\begin{cond}\label{cond:12.6.2.21.13bbb}
In addition to Condition~\ref{cond:12.6.2.21.13b} 
there exist splittings $q_1=q_{11}+q_{12}$ and 
$q_2=q_{21}+q_{22}$ 
by real-valued functions
\begin{align*}
q_{11}\in C^2(M)\cap L^\infty(M),\quad
q_{12}, 
q_{21}\in C^1(M)\cap L^\infty(M),\quad
q_{22}\in L^\infty(M)
\end{align*}
and constants $\rho,C>0$
such that for $k=0,1$
\begin{align*}
|\nabla^rq_{11}|&\le Cr^{-(1+\rho/2)/2},&
|\ell^{\bullet i}\nabla_i q_{11}|&\le Cr^{-1-\rho/2},&
|\mathrm d\nabla^rq_{11}|&\le Cr^{-1-\rho/2},\\
|\mathrm dq_{12}|&\le Cr^{-1-\rho/2},&
|(\nabla^r)^k q_{21}|&\le Cr^{-k-\rho},&
q_{21}\nabla^r q_{11}&\leq Cr^{-1-\rho},\\
|q_{22}|&\le Cr^{-1-\rho/2}.&&
\end{align*}
\end{cond}

Our  radiation condition bounds are
stated in terms of the  distributional radial differential operator 
$A$ defined in \eqref{eq:13.9.23.2.24}
and an asymptotic complex phase $a$ given below.
Pick a
  smooth decreasing function $r_\lambda\geq 2r_0$ of
  $\lambda>\lambda_0$ such that
  \begin{align}
    \lambda+\lambda_0-2q_1\geq 0\mfor r\geq r_\lambda/2,
  \label{eq:14.6.29.23.46}
  \end{align}
  and that  $r_\lambda=2r_0$ for all $\lambda$ large enough.
Then we set 
\begin{align*}
\eta_\lambda=1-\chi(2r/r_\lambda),
\end{align*}
and for $z=\lambda \pm \i \Gamma \in\mathcal I\cup\mathcal I_\pm$
\begin{subequations}
 \begin{align}
b&=\eta_\lambda|\mathrm dr|\sqrt{2(z-q_1)},\qquad\quad\ \  
\tilde b=\tilde\eta b,\label{eq:13.9.5.7.2300}\\
a&=b\pm\tfrac14\eta_\lambda (p^rq_{11})\big/(z-q_1),\label{eq:13.9.5.7.23}
\end{align} 
\end{subequations}
respectively, where the branch of square root is chosen such that 
$\mathop{\mathrm{Re}}\sqrt w>0$ for $w\in \mathbb C\setminus
(-\infty,0]$. Note that for $z\in \mathcal I$ there are two values of
$a$  which could be denoted $a_\pm$. For convenience we prefer to use
the shorter notation. 
Note also that the phase $a$ of \eqref{eq:13.9.5.7.23} 
is an approximate solution to the radial Riccati equation
\begin{align}
\pm p^ra+a^2-2|\mathrm dr|^2(z-q_1)=0
\label{eq:15.3.11.19.35}
\end{align} 
in the sense that it makes the quantity on the left-hand side of
\eqref{eq:15.3.11.19.35} small for large $r\ge 1$.  
The quantity $b$ of \eqref{eq:13.9.5.7.2300} alone already gives an
approximate solution to the same equation, however with the second
term of \eqref{eq:13.9.5.7.23} a better approximation is obtained. 
Set
\begin{align}
\beta_c
=\tfrac12\min\{\sigma,\tau,\rho\}.\label{eq:10c}
\end{align} Here and henceforth we consider $\sigma \in (0,\sigma')$
as a fixed parameter. 
\begin{thm}
  \label{prop:radiation-conditions} 
  Suppose Condition~\ref{cond:12.6.2.21.13bbb}, and let $I\subseteq \mathcal I$
be a compact interval.
  Then for all  $\beta\in [0,\beta_c)$
  there exists $C>0$ such that 
  for any $\phi=R(z)\psi$ with $\psi\in r^{-\beta}B$ and $z\in I_\pm$
\begin{align}
\|r^\beta(A\mp a)\phi\|_{B^*}
+\langle p_i^*r^{2\beta} h^{ij}p_j\rangle_{\phi}^{1/2}
&\leq C\|r^\beta\psi\|_B,\label{eq:14cccCff}
\end{align} 
respectively.
\end{thm}

The limiting absorption principle reads.

\begin{corollary}\label{cor:12.7.2.7.9b}
Suppose Condition~\ref{cond:12.6.2.21.13bbb},
and let $I\subseteq \mathcal I$ be a compact interval.
For any $s>1/2$
and $\epsilon\in (0,\min\{(2s-1)/(2s+1),\beta_c/(\beta_c+1)\})$ 
there exists $C>0$ such that for $k=0,1$ and any $z,z'\in I_+$ or $z,z'\in I_-$ 
\begin{align}
\|p^k R(z)-p^k R(z')\|_{\mathcal B(\mathcal H_s,\mathcal H_{-s})}\le C|z-z'|^\epsilon.
\label{eq:14.12.30.21.52}
\end{align}
In particular, the operators $p^k R(z)$, $k=0,1$,  attain
uniform limits as $I_\pm \ni z \to \lambda \in I$ in the norm topology of 
${\mathcal B}(\mathcal H_s,\mathcal H_{-s})$, say denoted  
\begin{align}
p^k R(\lambda\pm\mathrm i0):=\lim_{I_\pm \ni z\to \lambda}p^k R(z), \quad \lambda\in I,
\label{eq:14.12.30.21.53}
\end{align}
respectively. 
These limits $p^k R(\lambda\pm\mathrm i0)\in{\mathcal B}(B,B^*)$,
and  $R(\lambda\pm\mathrm i0):B\to \vN\cap B^*$.
\end{corollary}

Given  the limiting resolvents $R(\lambda\pm\mathrm i0)$
  the radiation condition bounds for real spectral parameters
follow directly from Theorem~\ref{prop:radiation-conditions}.

\begin{corollary}
  \label{cor:radiation-conditions}
  Suppose Condition~\ref{cond:12.6.2.21.13bbb},  and let 
  $I\subseteq \mathcal I$ be a compact interval.
  Then for all $\beta \in [0,\beta_c)$
  there exists $C>0$ such that 
  for any $\phi=R(\lambda\pm\mathrm i0)\psi$ with $\psi\in r^{-\beta}B$ and 
  $\lambda\in I$ 
\begin{align}
\|r^\beta(A\mp a)\phi\|_{B^*}
+\langle p_i^*r^{2\beta} h^{ij}p_j\rangle_{\phi}^{1/2}
&\leq C\|r^\beta\psi\|_B,\label{eq:14cccCa} 
 \end{align} 
respectively.
\end{corollary}

For the Euclidean and the hyperbolic spaces without potential $V$ we
can assume  
$\beta_c \ge 1-\epsilon$ for any (small) $\epsilon >0$
  (in fact since \eqref{eq:13.9.5.3.30C} is fulfilled with
    $\sigma=\sigma'$ in these cases we can assume  $\beta_c \ge 1$).

As another application of the radiation condition bounds 
we have  characterized  the limiting resolvents $R(\lambda\pm\mathrm i0)$. 
For the Euclidean space such characterization
is usually referred to as the Sommerfeld uniqueness result. 
\begin{corollary}\label{cor:13.9.9.8.23}
  Suppose Condition \ref{cond:12.6.2.21.13bbb}, and let
  $\lambda\in\mathcal I$, $\phi\in L^2_{\mathrm{loc}}(M)$ and
  $\psi\in r^{-\beta}B$ with  $\beta\in [0,\beta_c)$.
Then 
$\phi=R(\lambda\pm\mathrm i0)\psi$ holds if and only if
both of the following conditions hold:
\begin{enumerate}[(i)]
\item\label{item:13.7.29.0.29}
$(H-\lambda)\phi=\psi$ in the distributional sense.
\item\label{item:13.7.29.0.28}
$\phi\in \mathcal N\cap
r^\beta B^*$  and $(A\mp a)\phi\in r^{-\beta}B^*_0$.
\end{enumerate}
\end{corollary}

\subsection{Extended framework}

Let $r\geq r_0$, $\d \mathcal A_r$ 
be the naturally induced measure on $S_r$ and 
  \begin{align}
  \mathcal G_r=L^2(S_r,\mathrm d\tilde{\mathcal A}_r);\quad 
  \d \tilde{\mathcal A}_r 
  =
  |\mathrm dr|^{-1} \d \mathcal A_r.
\label{eq:13.9.20.0.17}
  \end{align}
 
Recall the
co-area formula  implying that for all integrable functions $\phi$ supported in $E$
\begin{equation}
  \label{eq:co_area}
  \int_E \phi(x) \bigl(\det g(x)\bigr)^{1/2} \,\d x
=\int_{r_0}^\infty \d r \int_{S_r}\phi
  \,\d \tilde{\mathcal A}_r,
\end{equation} in particular that for square integrable functions
\begin{equation*}
  \|1_E \phi\|^2=\int_{r_0}^\infty \|\phi_{|S_r}\|_{\mathcal G_r}^2\,\d r.
\end{equation*}
 We can describe the measure $\d \tilde{\mathcal A}_r $ in some details using
 the following condition.

\begin{cond}\label{cond:altG}
  Let $(M, g)$ be the manifold, $r$  be the
  function and $c$ and $r_0$ be the constants  of Condition
  \ref{cond:12.6.2.21.13}. Let $M_0=
\{x\in M\,|\, r(x)> r_0/2\}$. 
There exists a Riemannian manifold $(M^{\rm ex}, g^{\rm
    ex})$ of dimension $d$ in which the manifold $(M_0, g)$ is isometrically embedded. There exists an
  extension $r^{\rm ex}\in C^\infty(M^{\rm ex})$ of the restriction
  $r_{|M_0}$ 
    such that the extended vector field
  $\omega^{\rm ex}:=\grad r^{\rm ex}$ is  complete in $M^{\rm
    ex}$ and $|\omega^{\rm ex}|\ge c$  on $\{x\in M^{\rm ex}\,|\, r^{\rm ex}(x)> r_0/2\}$. Let $\tilde\omega^{\rm
    ex}=\tilde \eta^{\rm ex} \omega^{\rm ex}$ be the complete vector
  field defined with $\tilde\eta^{\rm ex}=|\mathrm \omega^{\rm
    ex}|^{-2}\bigl(1-\chi(2r^{\rm ex}/r_0)$, and let $\tilde y^{\rm
    ex}(t,\cdot)=\exp(t\tilde\omega^{\rm ex})$ denote the
  corresponding flow. Then
  \begin{align*}
    \forall \sigma\in S: \{\tilde y^{\rm ex}(t,\sigma)\,|\,t\ge
    0\}\cap M\neq\emptyset,
  \end{align*} where $S=S_{r_0}^{\rm ex}=\{ x\in M^{\rm ex}|\,r^{\rm
    ex}(x) =r_0\}$.
\end{cond} 
\begin{remark}\label{remark:extended-framewoM}
  The reader might prefer to think about $(M^{\rm ex}, g^{\rm ex})$
  and $r^{\rm ex}$ as given from the outset. Then $M_0\subseteq M^{\rm
    ex}$ would be a subset invariant under the forward flow of the
  vector field $\omega^{\rm ex}$. However since almost all of our
  conditions are needed for $M$ only (actually
  \eqref{eq:13.9.23.16.41b2} is the only quantitative exception) we
  have pursued the given presentation. For many examples, cf.
  \cite[Subsection~1.2]{IS2}, the vector field $\omega$ of Condition
  \ref{cond:12.6.2.21.13} is forward as well as backward complete
  (i.e. complete) and we can take $(M^{\rm ex}, g^{\rm ex}, r^{\rm
    ex})=(M, g,r)$. The typical origin for non-backward completeness
  for a sub-manifold $M\subseteq M'$, $M$ open in $M'$, is `crossing'
  of integral curves of $\omega$ at the boundary $\partial M\subseteq
  M'$.
\end{remark}

 We note
\begin{align}
\forall \sigma\in S\,\forall t\geq 0:\quad r^{\rm ex}(\tilde y^{\rm ex}(t,\sigma))=r_0+t,
\label{eq:13.9.21.16.44}
\end{align} and that any $x\in E^\mathrm{ex}:=\{x\in M^\mathrm{ex}\,|\,
r^\mathrm{ex}(x)>r_0\}$   has \textit{spherical coordinates}
defined as
\begin{align*}
  (r,\sigma)=(r^{\rm ex}(x), \tilde y^{\rm ex}(r_0-r^{\rm ex}(x),x))\in
  (r_0,\infty)\times S.
\end{align*} In particular any $x\in E$  has spherical coordinates
defined this way.  Mimicking the constructions \eqref{eq:13.9.20.0.17} we
introduce
\begin{align}
  \mathcal G=L^2(S,\mathrm d\tilde{\mathcal A}),\quad 
  \d \tilde{\mathcal A}=\d \tilde{\mathcal A}^{\rm ex} 
  =
  |\omega^{\rm
    ex}|^{-1} \d \mathcal A^{\rm ex},
\label{eq:13.9.20.0.172}
  \end{align} in terms of  the naturally induced measure $\d \mathcal
  A^{\rm ex}$. 
 Now, indeed   in spherical coordinates
 \begin{align*}
    \d \tilde{\mathcal A}_r =\exp
    \left(\int_{r_0}^r (\mathop{\mathrm{div}}\tilde\omega ^{\rm
    ex})(s,\sigma)\,\mathrm{d}s\right)\d \tilde{\mathcal
  A}\quad \text{for }x=\tilde y^{\rm ex}(r-r_0,\sigma)\in S_r.
 \end{align*} This leads to the isometrical embedding $\vG_r\subseteq \vG$, $r\geq r_0$,
given by mapping $\vG_r\ni \xi_r\to \xi^{\rm ex}\in \vG$ where 
\begin{align}\label{eq:trans2}
  \xi^{\rm ex}(\sigma)=\begin{cases}
    \exp
    \left(\int_{r_0}^r \tfrac 12(\mathop{\mathrm{div}}\tilde\omega ^{\rm
    ex})(s,\sigma)\,\mathrm{d}s\right)
\xi_r(x) 
&\text{ for }x=\tilde y^{\rm ex}(r-r_0,\sigma)\in S_r,\\
 0&\text{ otherwise}
  \end{cases}.
 \end{align} 

 The formula \eqref{eq:trans2} can be understood in terms of (a group
 of) 
 translations on the extended Hilbert space
 $\vH^\mathrm{ex}=L^2(M^{\rm ex}, g^{\rm ex})$.  We introduce the 
 \textit{normalized extended radial translation} $\tilde T^{\rm
   ex}(\tau)\ \colon {\mathcal H}^\mathrm{ex}\to{\mathcal
   H}^\mathrm{ex}$, $\tau\in \mathbb R$, in terms of the self-adjoint
 operator
 \begin{align*}
    \tilde A=\tilde A^\mathrm{ex}=\mathop{\mathrm{Re}}\parb{-\mathrm i
   \nabla_{\tilde \omega^\mathrm{ex}}}
 \end{align*}
  by $\tilde T^{\rm ex}(\tau)=
 \e^{\i \tau \tilde A}$. Then  \eqref{eq:trans2} is naturally 
 rewritten as $\xi^{\rm ex}= \e^{\i (r-r_0) \tilde A}\xi_r$ since 
 for  $\psi\in\mathcal H^{\rm ex}$ and $x\in M^{\rm ex}$
\begin{align*}
(\tilde T^{\rm ex}(\tau)\psi)(x)
=\exp \left(\int_0^\tau\tfrac12(\mathop{\mathrm{div}}\tilde\omega^{\rm ex})(\tilde y^{\rm ex}(t,x))\,\mathrm{d}t\right)\psi(\tilde y^{\rm ex}(\tau,x)).
\end{align*} 

We also note that the relation $x=\tilde y^{\rm ex}(r-r_0,\sigma)$
of \eqref{eq:trans2} naturally defines an embedding $S_r\subseteq S$
given as the map $ S_r\ni x\to \sigma\in S$. We shall sometimes
slightly abuse
notation and write $\sigma\in S_r$,  leaving it to the reader to
decide from the context  whether
$\sigma$  should be thought of as a point in
the subset $S_r$ of $M$ or rather as a point
 in the image
of this map.

\subsection{Review of results from \cite{IS3}}\label{subsec:1609141550}

We need  additional assumptions. The following one
suffices for constructing the \textit{distorted Fourier transform}. 
\begin{cond}\label{cond:12.6.2.21.13c}
  Along with Condition~\ref{cond:altG},
  Condition~\ref{cond:12.6.2.21.13bbb} holds  with
  \begin{align}\label{eq:BAScond}
    2\beta_c=\min\{\sigma,\tau,\rho\}>1.
  \end{align}
{In addition, the}
  functions $\tilde b=\tilde b(\lambda,x)$ and $q_1(x)$ have  real $C^1$-extensions
  to $\vI \times M^{\rm ex}$ and $M^{\rm ex}$,  respectively,  say denoted by ${\tilde b}^{\rm ex}$ and $q_1^{\rm ex}$ 
  (or for short
  by $\tilde b$ and $q_1$ again), {and the following
  bound holds uniformly in the spherical coordinates on $E$ and locally uniformly in $\lambda\in\vI$: 
\begin{align}
   \sup_{r_0\leq \check r\leq r}
\bigg|\ell^{\bullet i}\nabla_i\int_{\check r}^r\tilde b^{\rm ex}(s,\sigma)
\,\d s\bigg | 
\le 
C r^{-1/2}.
  \label{eq:13.9.23.16.41b2}
  \end{align}}
\end{cond}

We remark that if $M^{\rm ex}=M$ the technical bound
\eqref{eq:13.9.23.16.41b2} is a consequence of (parts of) the other
conditions and the bound \eqref{eq:13:9:34:16:20bb}. The extension of
$q_1$ is needed for Dollard type potentials discussed in Subsections
\ref{subsec:Short-range and Dollard classes of perturbations} and \ref{subsec:17053116}.

For any $\psi\in\mathcal H_{1+}$ and $r\geq r_0$ 
we introduce a function $\xi(r)\in \vG $ using
the mapping \eqref{eq:trans2}, omitting here (and often henceforth)  the superscript `ex':
\begin{align}\label{eq:defRes}
  \xi(r)(\sigma)= \exp \biggl(\int_{r_0}^{r}\parbb{\mp \i\tilde
    b+ \tfrac
    12\mathop{\mathrm{div}}\tilde\omega}(s,\sigma)\,\mathrm{d}s\biggr)
[\sqrt b R(\lambda\pm\mathrm
  i0)\psi](r,\sigma),
\end{align} (and $=0$ for $\sigma\notin S_r$) or, alternatively, 
\begin{align}
  \xi(r)=\e^{\i
  (r-r_0)(\tilde A^{\rm ex }\mp\tilde
      b^{\rm ex })}\bigl [\sqrt b R(\lambda\pm\mathrm
  i0)\psi\bigr
    ]_{|S_r}=\e^{\i
  (r-r_0)(\tilde A\mp\tilde
      b)}\bigl [\sqrt b R(\lambda\pm\mathrm
  i0)\psi\bigr
    ]_{|S_r}.
\label{eq:160126}
\end{align}
The $\vG$-valued function $\xi$ har a limit for $r\to \infty$ allowing
us to   define the ``distorted Fourier transform'' by
\begin{align}\label{eq:disF}
  F^\pm(\lambda)\psi =\vGlim_{r\to \infty} \xi(r);\quad 
\psi\in\mathcal H_{1+}.
\end{align}

  \begin{thm}
\label{thm:strong} 
Suppose Condition~\ref{cond:12.6.2.21.13c}.  Then for any $\psi\in
{\mathcal H_{1+}}$  there exist the limits
\eqref{eq:disF}. The maps $\mathcal
I\ni\lambda\mapsto F^\pm(\lambda)\psi\in \mathcal G$ are
   continuous. Moreover, putting $\delta(H-\lambda)=\pi^{-1}\mathop{\mathrm{Im}}R(\lambda+\i 0)$,
  \begin{align}
  \label{eq:fund}
  \|F^\pm(\lambda)\psi\|^2=2\pi\inp{\psi, \delta(H-\lambda)\psi}.
\end{align}
  \end{thm}

 By definition the function  $F^\pm(\lambda)\psi\in \vG=L^2(S,\d \tilde \vA)$, and we
note that  our  construction of
$F^\pm(\lambda)\psi$ is 
non-canonical   primarily due to the freedom in choosing $\vG$. In fact for
$M^{\rm ex}=M$ the only non-canonical feature comes from  the   dependence
of $r_0$ (determining  $\vG$ in that case), while in general there is an additional freedom in choosing
extended functions.

  Due to \eqref{eq:fund} the operators $F^\pm(\lambda)$ extend as
  continuous operators $B\to \vG$, and for any $\psi\in B$ the maps
  $F^\pm({}\cdot{})\psi\in \mathcal G$ are continuous. In Proposition
  \ref{prop:dist-four-transf}  stated below we give a formula for these
  extensions.  

Introduce 
\begin{align}
\mathcal H_{\mathcal I}=P_H(\mathcal I)\mathcal H,\quad 
  \widetilde \vH_{\mathcal I} =L^2(\mathcal I, (2\pi)^{-1}\d \lambda;\vG),
\label{eq:16091616}
\end{align} 
set $H_\mathcal I=H P_H(\mathcal I)$ and let $M_\lambda$ be the operator of multiplication by $\lambda$ on
$\widetilde\vH_{\mathcal I}$.
We define
\begin{align*}
  F^\pm=\int_{\mathcal I} \oplus F^\pm(\lambda)\,\d \lambda\colon
  B\to C(\mathcal I;\mathcal G).
\end{align*} 
These  operators can be extended to proper spaces which is stated as
the first  part of
the following result.
\begin{proposition}
  \label{prop:dist-four-transf} 
  Suppose Condition~\ref{cond:12.6.2.21.13c}. 
The operators   $F^\pm$ considered as  maps $B\cap\mathcal
H_{\mathcal I}\to \widetilde{\mathcal H}_{\mathcal I}$ extend uniquely
to  
isometries $\mathcal H_{\mathcal
  I}\to \widetilde{\mathcal H}_{\mathcal I}$. These  extensions obey $F^\pm H_\mathcal I\subseteq
M_\lambda F^\pm$.
Moreover  for  any $\psi\in B$ the vectors $F^\pm(\lambda)\psi $ are given as 
averaged limits. More precisely  introducing for any such $\psi$ the integral
${-\!\!\!\!\!\int_R}\xi(r)\,\d r:=R^{-1}\int_{R}^{2R} \xi(r)
\,\d r$, these  vectors  are given as
\begin{align}\label{eq:extbF2}
\begin{split} &F^\pm(\lambda)\psi =\vGlim_{R\to
    \infty} {-\!\!\!\!\!\!\int_R} \xi(r)\,\d r
\\&=\vGlim_{R\to
    \infty}-\!\!\!\!\!\!\int_R \exp \biggl( \int_{r_0}^{r}\parbb{\mp \i\tilde
    b+ \tfrac
    12\mathop{\mathrm{div}}\tilde\omega}(s,{}\cdot{})\,\mathrm{d}s\biggr)
[\sqrt b R(\lambda\pm\mathrm
  i0)\psi](r,{}\cdot{})
\,\d r,
\end{split}
\end{align} and the limits
\eqref{eq:extbF2} are attained locally
uniformly in $\lambda\in\mathcal I$.
\end{proposition}

The above extended isometries 
$F^\pm\colon \mathcal H_{\mathcal I}\to\widetilde{\mathcal H}_{\mathcal I}$
are actually unitary under an additional condition,
and for  this  reason 
we call them the Fourier transforms associated with $H_{\mathcal I}$. 
This  condition consists of two alternatives. The first one is a partial 
  strengthening of Condition~\ref{cond:12.6.2.21.13c}. The other one
  is primarily a set of
   bounds on  higher order derivatives of various quantities defined on $M$.
   
{For simplicity for any smooth function $f$ on $M$ let us set 
\begin{align}
\nabla' f=\nabla f-(\nabla_{\tilde \omega}f)\nabla r,\quad 
\nabla'{}^2f=\nabla^2f-(\nabla_{\tilde \omega} f)\nabla^2r.
\label{eq:176169}
\end{align} 
In $E$ 
the spherical parts of $\nabla'f$ and $\nabla'{}^2f$,
i.e.  $\ell^{\bullet i}(\nabla'f)_i$ and 
$\ell^{\bullet i}\ell^{\bullet j}(\nabla'{^2}f)_{ij}$,
coincide with first and second order derivatives computed by 
the Levi--Civita connections on the $r$-spheres $S_{r}$ 
associated with the induced Riemannian metrics $g_r:=\iota^*_rg$.}

\begin{cond}\label{cond:12.6.2.21.13bb}
In addition to  Condition~\ref{cond:12.6.2.21.13c} with the extension
${\tilde b}^{\rm ex}$ being   $C^2$ one of the
following properties holds:
\begin{enumerate}
\item\label{item:14.5.1.8.30} 
\begin{align}
\min\{\sigma,\tau,\rho\}>2.\label{eq:13.9.12.13.46bB}
\end{align}
\item\label{item:14.5.1.8.31}
{
The restriction $q_1{}_{|E}$ 
belongs to $C^2(E)$,  and 
 there exists  $C>0$ such that
\begin{subequations}
\begin{align}
 \bigl|\ell^{\bullet i}\ell^{\bullet j}\ell^{\bullet k}(\nabla^3r)_{ijk}\bigr|&\le
Cr^{-1-\tau/2},\label{eq:smootrA}\\
|\ell^{\bullet i}\ell^{\bullet j}(\nabla'{}^2 q_1)_{ij}|&\le Cr^{-1-\rho},\label{eq:smootV}
\end{align}
and 
\begin{align}
\label{eq:smootr0}      
\bigl|\ell^{\bullet i}\ell^{\bullet j}(\nabla'{}^2|\mathrm dr|^2)_{ij}\bigr| 
&\leq Cr^{-1-\tau},\qquad 
  \bigl|\ell^{\bullet i}\ell^{\bullet j}(\nabla'{}^2\nabla^r|\mathrm dr|^2)_{ij}\bigr|\le
  Cr^{-1-\tau},\\
  \bigl|\ell^{\bullet i}\ell^{\bullet j}(\nabla'{}^2\Delta r)_{ij}\bigr|&\le Cr^{-1-\tau}.\label{eq:smootr}
\end{align}
\end{subequations}}
\end{enumerate}
\end{cond}
{We} also remark that the additional
smoothness condition on ${\tilde b}^{\rm ex}$ in the case of
\eqref{item:14.5.1.8.30} is needed only in the proof of Lemma
\ref{lem:14.5.8.13.50} (actually it is only the smoothness in $\lambda$
that is used there).

\begin{thm}
  \label{thm:dist-four-transf} 
Suppose Condition~\ref{cond:12.6.2.21.13bb}. 
Then  the operators  
$F^\pm\colon \mathcal H_{\mathcal I}\to \widetilde{\mathcal H}_{\mathcal I}$
are unitarily diagonalizing transforms for $H_{\mathcal I}$,
that is, they are unitary and
 \begin{align*}
   F^\pm H_{\mathcal I}=M_\lambda F^\pm,
\end{align*}
respectively.
\end{thm}

Under Condition~\ref{cond:12.6.2.21.13bb} and for any $\lambda\in\mathcal I$ the {\it scattering matrix} 
$S(\lambda)\colon \mathcal G\to\mathcal G$ is defined by the
identity 
\begin{align}
    \label{eq:1S} F^+(\lambda)\psi=S(\lambda)F^-(\lambda)\psi;\quad\psi\in  B.
\end{align}
It follows from \cite{IS3}   that 
$C^\infty_{\mathrm c}(S)\subseteq \mathop{\mathrm{Ran}}
F^\pm(\lambda)$, and hence, with
Theorem \ref{thm:strong}, Proposition~\ref{prop:dist-four-transf} and a density argument,
$S(\cdot)$ is a well-defined strongly continuous unitary operator.  We
state below  a
 characterization of the generalized eigenfunctions in $\mathcal N
 \cap  B^*$, i.e. the elements of 
\begin{equation*}
    \vE_\lambda:=\{\phi\in\mathcal N \cap  B^*\,|\, (H-\lambda)\phi=0\}.
  \end{equation*} Due to Theorem \ref{thm:13.6.20.0.10} these
  eigenfunctions  may be called \emph{minimum}. We 
introduce for any $\xi\in
  \mathcal G$  purely outgoing/incoming  approximate  generalized eigenfunctions  $\phi^\pm[\xi]=\phi^\pm_\lambda[\xi]\in
B^*$ in terms of the spherical coordinates by  
\begin{align}\label{eq:gen1B}
\begin{split}
   \phi^\pm[\xi](r,\sigma)&
=\eta_\lambda [2|\mathrm dr|^2 (\lambda-q_1)]^{-1/4} 
\exp \biggl(\int_{r_0}^{r}\parbb{\pm \i\tilde
    b- \tfrac
    12\mathop{\mathrm{div}}\tilde\omega}(s,\sigma)
\,\mathrm{d}s\biggr)\xi(\sigma).
\end{split}
\end{align} 
Of course these quantities are well-defined independently of all the
estimates of 
Conditions \ref{cond:12.6.2.21.13c} and \ref{cond:12.6.2.21.13bb}.  We remark that formulas like
\eqref{eq:gen1B} in the context of Schr\"odinger operators are
referred to as (zeroth order) WKB-approximations. 

\begin{thm}
  \label{thm:char-gener-eigenf-1}
Suppose Condition~\ref{cond:12.6.2.21.13bb}.
Then for any $\lambda \in\mathcal I$  the following assertions hold.
\begin{subequations}
  \begin{enumerate}[(i)]
  \item\label{item:14.5.13.5.40} For any one of $\xi_\pm \in \vG$ or $\phi\in \vE_\lambda$ the
    two  other  quantities in $\{\xi_-,\xi_+,  \phi\}$ uniquely exist
    such that
    \begin{align}\label{eq:gen1}
      \phi -\phi^+[\xi_+]+\phi^-[\xi_-]\in
      B_0^*.
    \end{align}

  \item \label{item:14.5.13.5.41} The  correspondences  in  \eqref{eq:gen1} are  given  by  the
    formulas (recall \eqref{eq:160126})
    \begin{align}\label{eq:aEigenfw}
     \phi&=\i F^\pm(\lambda)^*\xi_\pm,\quad \xi_+=S(\lambda)\xi_-,\\
\xi_\pm&=2^{-1} \vGlim_{R\to \infty}
-\!\!\!\!\!\!\int_R \e^{\i(r-r_0)(\tilde A^{\rm ex }\mp\tilde
      b^{\rm ex })}
\bigl [{b^{-1/2}}(  A\pm b)\phi\bigr]_{|S_r}\,\d r.\label{eq:aEigenfwB}
    \end{align}
    In  particular  the  wave  matrices  $F^\pm(\lambda)^*\colon\vG\to
    \vE_\lambda$ are linear isomorphisms.

  \item\label{item:14.5.13.5.42}  The wave matrices
    $F^\pm(\lambda)^*\colon\vG\to \vE_\lambda\,(\subseteq B^*)$
    are bi-continuous. In fact
    \begin{align}\label{eq:aEigenf2w}
      2\|\xi_\pm\|_{\vG}^2=\lim_{R\to \infty}R^{-1}\int_{B_{2R}\setminus B_{R}}
      |{b}^{1/2}\phi|^2\,(\det g)^{1/2}\d x.
    \end{align}

  \item\label{item:14.5.14.4.17}   The   operators    $F^\pm(\lambda)\colon   B\to   \vG$   and
    $\delta(H-\lambda)\colon B\to \vE_\lambda$ are onto.
  \end{enumerate}
   \end{subequations}

\end{thm}

Finally we give  an  application of our results to 
channel scattering theory addressed,  but
treated very differently,  in \cite{HPW}.
Suppose $M^\mathrm{ex}$ has $N\geq 2$ number of ends, i.e.\  
$E^\mathrm{ex}=\{x\in M^\mathrm{ex}\,|\, r^\mathrm{ex}(x)>r_0\}$ has
$N\geq 2$ components $E_i$, $i=1,\dots,N$. (Note that this implies that
$ E_i\cap M$, $i=1,\dots,N$, are the components of $E= E^\mathrm{ex}\cap M$.)
Then the Hilbert space $\mathcal G$ splits as 
\begin{align*}
\mathcal G=\mathcal G_1\oplus\dots\oplus\mathcal G_N;\quad \mathcal G_i=L^2(S_i),\ 
S_i=S\cap \overline{E_i},
\end{align*}
and, accordingly, the scattering matrix $S(\lambda)$ has a matrix representation
\begin{align*}
S(\lambda)=(S_{ij}(\lambda))_{1\leq i,j\leq N},\quad
S_{ij}(\lambda)\in \mathcal B(\mathcal G_j,\mathcal G_i).
\end{align*}
\begin{corollary}\label{cor:transmission}
Suppose Condition~\ref{cond:12.6.2.21.13bb},
and that $E^\mathrm{ex}$ has $N$ number of ends.  
For any $\lambda\in\mathcal I$
let the scattering matrix $S(\lambda)$ be 
decomposed into components as above. 
Then the off-diagonal components, $S_{ij}(\lambda)$ with  $i\neq j$, are
  one-to-one mappings.
\end{corollary}

We note that Corollary \ref{cor:transmission} may be seen as a
stationary solution to conjectures of \cite{HPW}, see \cite[Remark
5.7]{HPW}. We shall develop the time-dependent version of this
result,  more directly
addressing  conjectures of \cite{HPW}.

\section{Main results}\label{subsec:result}

We present our results deferring the main parts of the proofs  to Section~\ref{subsec:Time-dependent wave operator}.

\subsection{General results for long-range models}

According to the arguments of Subsection~\ref{subsec:1609141550}
we introduce the following \textit{free comparison dynamics} 
defined on $\widetilde{\mathcal H}_{\mathcal I}$ in \eqref{eq:16091616},
employing the purely outgoing/incoming approximate  
generalized eigenfunctions \eqref{eq:gen1B}:
For any $t\geq0$ and $h\in C^1_{\mathrm c}(\mathcal I\times S)
\subseteq \widetilde{\mathcal H}_{\mathcal I}$
we let 
\begin{align}
  \begin{split}
U^\pm(t)h
  =(\pm2\pi\i)^{-1}
  \int_{\mathcal I} \e^{\mp \i t\lambda}\phi^\pm_\lambda[h(\lambda,{}\cdot{})]\,\d \lambda.
\label{eq:14.3.12.20.22}  
  \end{split}
\end{align} 
In Lemma~\ref{lem:14.4.29.22.28b} we will see that $U^\pm(t)h\in
\mathcal H$ for any $t\geq 0$ and $h\in C^1_{\mathrm c}(\mathcal
I\times S)$. A consequence of Lemma \ref{lem:14.4.29.22.28} is that
these dynamics 
 are  asymptotically
isometric, i.e. $\lim_{t\to \infty}
 \|U^\pm(t)h\|_{\vH}=\|h\|_{\widetilde{\mathcal
     H}_{\mathcal I}}$. We also note the time reversal invariance property
 $\overline{U^+(t)h}=U^-(t)\bar h$.

\begin{thm}\label{thm:14.3.12.18.20}
Suppose Condition~\ref{cond:12.6.2.21.13bb}.
Then for any $h\in C^2_{\mathrm c}(\mathcal I\times S)$
there exist the  limits
\begin{align}
  \label{eq:53}
  W^\pm h:=\lim_{t\to \infty}\e^{\pm\i tH}U^\pm(t)h\quad \mbox{in }\mathcal H.
\end{align}
These limits $W^\pm$ extend uniquely to unitary  operators 
$W^\pm\colon \widetilde{\mathcal H}_{\mathcal I}\to\mathcal H_{\mathcal I}$,
and 
\begin{align}\label{eq:54}
  F^\pm=(W^\pm)^*,
\end{align}
respectively.
 Whence  the wave operators 
$W^\pm$  are complete on 
$\mathcal I$.
\end{thm}

In the proof of  Theorem~\ref{thm:14.3.12.18.20} 
  we will use a more simplified free dynamics, 
which comes about as  the leading term of \eqref{eq:14.3.12.20.22} 
extracted by the stationary phase argument, see Lemma~\ref{lem:14.4.29.22.28}. 
We also note that for so-called short-range and Dollard type
models there are further  simplified dynamics for which the
analogues of Theorem~\ref{thm:14.3.12.18.20} hold. 
This will be discussed later in this section.

Now let us apply Theorem~\ref{thm:14.3.12.18.20} to one aspect of  channel scattering theory.
In the $N$-ends setting as in Corollary \ref{cor:transmission} let us consider
the orthogonal projections
\begin{align}
  P^\pm_i=W^\pm 1_{\mathcal I\times S_i}(W^\pm)^*;\quad i=1,\dots,N.
    \label{eq160917}
\end{align}
Clearly these projections non-trivially resolve the 
space $\mathcal H_{\mathcal I}$: 
\begin{align*}
P^\pm_i\neq 0\text{ for all }i,\quad 
P^\pm_iP^\pm_j=0\text{ for all }i\neq j, \quad
P_1^\pm+\dots+P_N^\pm=I\text{ on }\vH_{\vI}.
\end{align*}
It is not difficult to see by the stationary phase method (or by using
\eqref{eq:twosidedest}, \eqref{eq:1609220}  and \eqref{eq:14.5.9.12.13BB} 
directly) that 
these projections have dynamical representations 
\begin{align*}
  P^\pm_i=\slim_{t\to \infty}e^{\pm \i tH}1_{E_i\cap M}e^{\mp \i
    tH}P_H(\vI),
\end{align*}
and hence we can interpret that $P^\pm_i$ describe the initial/final state
components of wave packets that 
go to/come from the $i$-th end $E_i\cap M$, respectively.
Then the following corollary says that any  wave packet coming from one end,
say the $j$-th end, always enters  all the others.

\begin{corollary}\label{cor:transmission2}
 Suppose the conditions of Corollary \ref{cor:transmission},
and let $P^\pm_i$, $i=1,\dots,N$, be defined by 
\eqref{eq160917}.
For any $i,j\in\{1,\dots,N\}$ with $i\neq j$
and for any nonzero $\psi\in P^-_j\vH_{\vI}$
one has $P^+_i\psi\neq0$.
\end{corollary}
\begin{proof}
Let $i\neq j$, and suppose $0\neq \psi\in P^-_j\vH_{\vI}$. 
By Theorem~\ref{thm:14.3.12.18.20} we have
\begin{align*}
P^+_i\psi
&=
P^+_iP^-_j\psi
\\&=W^+ 1_{\mathcal I\times S_i}F^+W^- 1_{\mathcal I\times S_j}(W^-)^*\psi
\\&
=W^+ 1_{\mathcal I\times S_i}S(\cdot)F^-W^- 1_{\mathcal I\times S_j}(W^-)^*\psi
\\&
=W^+ S_{ij}(\cdot)\parb{(W^-)^*\psi}_j,
\end{align*}
from which we can deduce that $P^+_i\psi\neq0$ thanks to Corollary
\ref{cor:transmission}. 
\end{proof}

In the above discussion and proof we identified    `dynamical
transmission' and `spectral transmission'. In the literature a similar
identification has been shown for reflection for $1$-dim Schr\"odinger operators, see \cite{LPPZ} for
a recent contribution.

\subsection{Short-range and Dollard classes of perturbations}
\label{subsec:Short-range and Dollard classes of perturbations}
The conditions in Section~\ref{subsec:preliminary result} are
satisfied for the free Euclidean and hyperbolic spaces, and hence also
for their perturbations to some extent.  In fact perturbations of
long-range types are allowed in Section~\ref{subsec:preliminary
  result}, and the results so far are the most general ones holding
for such long-range models. However in certain more restrictive
settings it is possible to show  versions of Theorem
\ref{thm:14.3.12.18.20} with  considerably simpler comparison dynamics and phase modifiers.

Here let us precisely formulate the notions of short-range and 
Dollard types of  the effective potential
$q$ assuming  Condition~\ref{cond:12.6.2.21.13c}.  We will discuss the corresponding simplifications in the
following subsections.

We introduce the following quantity using spherical coordinates
\begin{align*}
\lambda_0(\sigma)=\limsup_{r\to\infty}q_1(r,\sigma);\quad \sigma\in S.
\end{align*}  
Note that $\lambda_0(\sigma)\leq \lambda_0$ for all $\sigma\in S$.

\begin{subequations}
\begin{defn}\label{def:2014.12.10.7.17}
The effective potential $q=q_1+q_2$ of Condition~\ref{cond:10.6.1.16.24}
is said to be of \emph{short-range type}, 
if there exist $\epsilon,C>0$ such  that uniformly in the spherical
coordinates on $E$
\begin{align}\label{eq:shortcondib}
|q_1(r,\sigma)-\lambda_0(\sigma)|\le Cr^{-1-\epsilon}.
\end{align}
\end{defn}

Note the following consequence of $q$ being of short-range type: Locally uniformly in $\lambda \in \vI$ 
\begin{align}\label{eq:shortcondib2}
\lim_{R\to \infty }\sup_{\sigma\in S_R}\int^\infty_{R}|\tilde b_{\rm
      sr}-\tilde b|(s,\sigma)\,\mathrm{d}s=0;
\end{align} here 
$$\tilde b_{\rm sr}(r,\sigma) :=
  |\d r(r,\sigma)|^{-1}\sqrt{2(\lambda-\lambda_0(\sigma))}\quad
\text{for }r> r_0/2.$$
    Upon replacing
  $ \d r$
  by   $ \d r^{\rm ex}$ we shall use this definition on $E^{\rm ex}$ as well.
 \end{subequations}

The effective potential $q$ not necessarily of short-range type 
is said to be 
of \emph{long-range type}. 
The manifold $(M,g)$ is said to be of 
\emph{short-range} or 
\emph{long-range  type} 
if $q$ is of short-range or long-range type when $V=0$, respectively.
  
\begin{subequations}
\begin{defn}\label{def:2014.12.10.7.17B}
The effective potential $q=q_1+q_2$ of Condition~\ref{cond:10.6.1.16.24}
is said to be of \emph{Dollard type}, 
if there exist $\epsilon,C>0$ such  that uniformly in the spherical
coordinates on $E$
\begin{align}
|q_1(r,\sigma)-\lambda_0(\sigma)|\le Cr^{-(1+\epsilon)/2}.
\label{eq:dollard}
\end{align}
\end{defn}

Note the following consequence of $q$ being of Dollard type: Locally uniformly in $\lambda \in \vI$
\begin{align}\label{eq:dollard2}
\lim_{R\to \infty }\sup_{\sigma\in S_R}\int^\infty_{R}|\tilde b_{\rm
      do}-\tilde b|(s,\sigma)\,\mathrm{d}s=0,
\end{align} 
where 
\begin{align*}
\tilde b_{\rm do}(r,\sigma)
=\tilde b_{\rm sr}(r,\sigma)
\Big (1-\tfrac 12 \tfrac{q_1(r,\sigma)-\lambda_0(\sigma)}{\lambda-\lambda_0(\sigma)}\Big )
\quad\text{for }r> r_0.
\end{align*}  
  Upon replacing $
q_1$ by $
  q_1^{\rm ex}$ (the latter function introduced in
  Condition~\ref{cond:12.6.2.21.13c}) we shall use this definition on $E^{\rm ex}$ as well.
\end{subequations}

Clearly a $q$ of short-range type is also of Dollard type.
The Dollard type is a particular long-range type for which certain
approximations work. We will elaborate on this in  Corollary
\ref{cor:dollard} and Theorem \ref{thm:14.3.12.18.20dol}.  In the
Schr\"odinger operator literature this is usually accredited Dollard,
\cite{Do}. 

We remark that, if $M^{\rm ex}=M$, obviously \eqref{eq:shortcondib2} and
\eqref{eq:dollard2}  follow  from
\eqref{eq:shortcondib} and 
\eqref{eq:dollard}, respectively.

\subsection{Simplification for short-range models}\label{subsec:170531}

Here we discuss simplifications of the main result, Theorem~\ref{thm:14.3.12.18.20},
for a short-range model.
At the end of the subsection we also discuss a comparison with 
the setting of \cite{IS1}.

\subsubsection{Simplification}

Assume that the effective potential $q$ is 
of short-range type. 
For simplicity we also assume 
\begin{align}
|\mathrm d r^{\rm ex}|=1\quad\text{for } r> r_0/2.
\label{eq:1609189}
\end{align}  

Then we set 
\begin{align*}
b_{\mathrm{sr}}(\sigma)
=\tilde b_{\rm sr}(\sigma)
=\sqrt{2(\lambda-\lambda_0(\sigma))},
\end{align*}
and for any 
$\xi\in\mathcal G$ 
\begin{align}\label{eq:gen1Bb}
\begin{split}
   \phi^\pm_{\mathrm{sr}}[\xi](r,\sigma)&= \eta(r)
   b_{\mathrm{sr}}(\sigma)^{-1/2}
   \exp \biggl(
\pm\mathrm i b_{\mathrm{sr}}(\sigma)(r-r_0)
-\tfrac12\int_{r_0}^r\Delta r(s,\sigma)
\,\mathrm{d}s\biggr)\xi(\sigma),
\end{split}
\end{align}
cf.\ \eqref{eq:13.9.5.7.2300} and \eqref{eq:gen1B}.
Note that $b_{\mathrm{sr}}$ is the zeroth order approximation of 
$b$ as $r\to\infty$.

We first present a simplified version of the distorted Fourier transforms.
\begin{corollary}\label{cor:strongbb}
In addition to Condition~\ref{cond:12.6.2.21.13c},
suppose that $q$ is of short-range type,
and that \eqref{eq:1609189} holds.
\begin{enumerate}
\item
For any $\lambda\in\mathcal I$ and $\psi\in B$
there exist $F^\pm_{\mathrm{sr}}(\lambda)\psi\in \mathcal G$ such that 
\begin{align*}
R(\lambda\pm\mathrm i0)\psi-\phi^\pm_{\mathrm{sr}}
[F^\pm_{\mathrm{sr}}(\lambda)\psi]\in B^*_0.
\end{align*}
In addition, $F^\pm_{\mathrm{sr}}(\lambda)\psi$ 
are separately continuous in $\lambda\in \mathcal I$
and $\psi\in \mathcal G$.
\item
The operators 
\begin{align*}
F^\pm_{\mathrm{sr}}=\int_{\mathcal I} \oplus F^\pm_{\mathrm{sr}}(\lambda)\,\d \lambda
\end{align*} 
considered as mappings $B\cap\mathcal
H_{\mathcal I}\to \widetilde{\mathcal H}_{\mathcal I}$ 
extend uniquely
to  isometric  operators
$\mathcal H_{\mathcal I}\to \widetilde{\mathcal H}_{\mathcal I}$. 
\item
The above $F^\pm_{\mathrm{sr}}$ are related to $F^\pm$
as follows:  
\begin{align*}
    F_{\rm sr}^\pm=\e^{\mp\i \theta}F^\pm;\quad
    \theta(\lambda,\sigma) = \int^\infty_{r_0}( b_{\rm
      sr}- b)(s,\sigma)\,\mathrm{d}s.
     \end{align*}
\end{enumerate}
\end{corollary}
\begin{proof}
The assertions are  immediate consequences of the assumptions and 
our  previous results, that is more precisely, those  reviewed in Section~\ref{subsec:preliminary
  result} and \cite[{Lemma 3.13}]{IS3}.
\end{proof}

A simplified free dynamics for a short-range model 
may be constructed just by 
replacing $\phi^\pm$ by $\phi_{\mathrm{sr}}^\pm$ in
\eqref{eq:14.3.12.20.22},
but here let us go one step forward. 
According to the stationary phase argument, cf.\ Lemma~\ref{lem:14.4.29.22.28}, 
we should for $h\in \widetilde{\mathcal H}_{\mathcal I}$ and $t>0$ study  
\begin{align}
\begin{split}
U^\pm_{\mathrm{sr}}(t)h(r,\sigma)
=&(2\pi )^{-1/2}{\mathrm e^{\mp 3\mathrm i\pi/4}}
1_{[r_0,\infty)}(r)
\mathrm e^{\pm\mathrm iK_{\mathrm{sr}}}
\e^{-\int_{r_0}^r(\Delta r)(s,\sigma)/2\,\mathrm ds}
\\&\phantom{{}={}}{}\cdot
\bigl(\tfrac{r-r_0}{t^2}\bigr)^{1/2}
h\bigl(\tfrac{(r-r_0)^2}{2t^2}+\lambda_0(\sigma),\sigma\bigr),
\end{split}
\label{eq:16091819}
\end{align}  where
\begin{align*}
  K_{\mathrm{sr}}=(r-r_0)^2/(2t)-t\lambda_0(\sigma).
\end{align*} 
 Note that $U^\pm_{\mathrm{sr}}(t)$ is a contraction,
 i.e. $\|U^\pm_{\mathrm{sr}}(t)h\|_{\vH}\leq
 \|h\|_{\widetilde{\mathcal H}_{\mathcal I}}$, and that the dynamics 
 $U^\pm_{\mathrm{sr}}(\cdot)$ are  asymptotically
 isometric, i.e. $\lim_{t\to \infty}
 \|U^\pm_{\mathrm{sr}}(t)h\|_{\vH}=\|h\|_{\widetilde{\mathcal
     H}_{\mathcal I}}$. In any case  Theorem~\ref{thm:14.3.12.18.20} simplifies  as follows.

\begin{thm}\label{thm:14.3.12.18.20sr}
In addition to Condition~\ref{cond:12.6.2.21.13bb},
suppose that $q$ is of short-range type,
and that \eqref{eq:1609189} holds.
For all $h\in \widetilde{\mathcal H}_{\mathcal I}$ 
there exist the limits
\begin{align}
  \label{eq:53sr}
  W_{\rm sr}^\pm h:=\lim_{t\to \infty}\e^{\pm\i tH}
U_{\rm sr}^\pm(t)h\quad \mbox{in } \mathcal H.
\end{align} 
These limits $W_{\rm sr}^\pm$ are unitary  operators 
$W_{\rm sr}^\pm\colon \widetilde{\mathcal H}_{\mathcal I}\to\mathcal H_{\mathcal I}$,
and  
  \begin{align}\label{eq:54sr}
  F_{\rm sr}^\pm=(W_{\rm sr}^\pm)^*,
\end{align}
respectively.  
Whence  the wave operators $W^\pm_{\mathrm{sr}}$ 
 are complete on $\mathcal I$.
\end{thm}

We note that the standard hyperbolic space fits into this framework 
with a positive constant $\lambda_0(\sigma)\equiv \lambda_0>0$,
and the Euclidean space does so with the zero constant $\lambda_0(\sigma)\equiv 0$.

\subsubsection{Comparison with \cite{IS1}}

Next we furthermore assume $\lambda_0(\sigma)\equiv 0$
and compare Theorem~\ref{thm:14.3.12.18.20sr}
with the main result of \cite{IS1} on  existence and completeness 
of  wave operators on asymptotically Euclidean manifolds.

We begin with a comparison of the settings.
For reference let us quote below the conditions of \cite{IS1} in
a form suitable for comparison.

\begin{cond}\label{cond:14.5.24.7.4}
Let $(M,g)$ be a connected and complete Riemannian manifold,
and let $V\in L^\infty(M)$ be real-valued.
There exist a function $r\in C^\infty(M)$ with image $r(M)=[1,\infty)$
and constants 
$\delta,\kappa,\eta,C>0$ and $r_0\ge 2$ such that:
\begin{enumerate}
\item
The $r$-balls $\{x\in M\,|\, r(x)<R\}$, $R>0$, are relatively compact in $M$.
\item
The following relations hold for $r(x)\ge r_0/2$:
  \begin{align*}
    |\mathrm dr|=1,\quad
    R\iota_R^* \nabla^2r\ge \tfrac12(1+\delta) \iota_R^*g. 
\end{align*}
\item
The following estimates hold globally on $M$ for $\alpha=0,1$:
  \begin{align}
r\ge 1,\quad 
    |\nabla^\alpha \Delta r|\le C r^{-1/2-\alpha-\kappa},\quad
    |V|\le Cr^{-1-\eta}. \label{eq:34b}
\end{align}
\end{enumerate}
\end{cond}

Condition~\ref{cond:14.5.24.7.4} above would seem quite different from 
the setting of \cite{IS1},
but obviously Condition~\ref{cond:14.5.24.7.4} follows from
Conditions~1.1--1.4 of \cite{IS1}. On the other hand,
Condition~\ref{cond:14.5.24.7.4} constitutes what is used in the
proofs of \cite{IS1}, and consequently the results of \cite{IS1} remain
valid under Condition~\ref{cond:14.5.24.7.4}. 
Hence we may regard Condition~\ref{cond:14.5.24.7.4} exactly as the 
setting of \cite{IS1}.

The setting of this paper may be considered 
slightly more restrictive than that of \cite{IS1}.
In fact, we have the following lemma (with obvious proof).
\begin{lemma}\label{lem:170601}
Suppose Condition~\ref{cond:14.5.24.7.4}  with a splitting $V=V_1+V_2$
such that 
$$| V_1|\le Cr^{-1-\eta},\quad |\nabla^r V_1|\le Cr^{-2-\eta},\quad |V_2|\le Cr^{-3/2-\eta}.$$
Then Condition \ref{cond:12.6.2.21.13c}  is satisfied for
$$q_{11}=0,\quad q_{12}=0,\quad
q_{21}=V_1+\tfrac18\tilde\eta(\Delta r)^2,\quad
q_{22}=V_2+\tfrac14\tilde\eta\nabla^r\Delta r.$$
 In addition, if \eqref{eq:smootrA} and \eqref{eq:smootr} hold, then 
Condition~\ref{cond:12.6.2.21.13bb} is also satisfied.
\end{lemma} 

If $\lambda_0(\sigma)\equiv 0$  
the simplified comparison dynamics 
$U_{\rm sr}^\pm(t)$
defined by \eqref{eq:16091819}
coincides with the comparison dynamics
of \cite{IS1}. 
This observation will allow us to employ the existence of  the wave
operator $\Omega_+$ from \cite{IS1}  
 to  deduce the assertions of 
Theorem~\ref{thm:14.3.12.18.20sr}
without imposing the stronger
Condition~\ref{cond:12.6.2.21.13bb}. In particular this shows a
stationary representation of $\Omega_+$. Whence we conclude
 the following alternative version of 
Theorem~\ref{thm:14.3.12.18.20sr}.

\begin{thm}\label{thm:14.3.12.18.20srb}
  Suppose Condition~\ref{cond:14.5.24.7.4} and the splitting
  assumption of Lemma \ref{lem:170601}.  For all $h\in
  \widetilde{\mathcal H}_{\mathcal I}$ there exist the limits
\begin{align}
  \label{eq:53srb}
  W_{\rm sr}^\pm h:=\lim_{t\to \infty}\e^{\pm\i tH}
U_{\rm sr}^\pm(t)h\quad \mbox{in } \mathcal H.
\end{align} 
These limits $W_{\rm sr}^\pm$ are  unitary  operators 
$W_{\rm sr}^\pm\colon \widetilde{\mathcal H}_{\mathcal I}\to\mathcal
H_{\mathcal I}$, and  
\begin{align}\label{eq:54srb}
  F_{\rm sr}^\pm=(W_{\rm sr}^\pm)^*.
\end{align} 
\end{thm}

Since Condition~\ref{cond:12.6.2.21.13bb} is missing, 
$F^\pm_{\mathrm{sr}}\colon \mathcal H_{\mathcal I}
\to \widetilde{\mathcal H}_{\mathcal I}$
can be constructed only as isometries by the arguments of \cite{IS2,IS3}.
However, 
we can deduce unitarity of $F^\pm_{\mathrm{sr}}$  from \eqref{eq:54srb} 
without Condition~\ref{cond:12.6.2.21.13bb}.

\subsection{Simplification for Dollard type  models}\label{subsec:17053116}

Similarly to the short-range case we can simplify 
our main result  Theorem~\ref{thm:14.3.12.18.20}  for  the Dollard
case  too.
Here we assume that $q$ is of Dollard type
and also that \eqref{eq:1609189} holds.
Then we set for $r> r_0$
\begin{align*}
b_{\rm do}(r,\sigma)= 
\tilde b_{\rm do}(r,\sigma)= \sqrt{2(\lambda-\lambda_0(\sigma))}
+\bigl[\lambda_0(\sigma)-q_1(r,\sigma)\bigr]\big/\sqrt{2(\lambda-\lambda_0(\sigma))},
\end{align*}
and for $\xi\in\mathcal G$
\begin{align*}
\begin{split}
   \phi^\pm_{\mathrm{do}}[\xi](r,\sigma)&= \eta_\lambda(r)
   b_{\mathrm{do}}(\sigma)^{-1/2}
   \exp \biggl(
\int_{r_0}^r \bigl(
\pm\mathrm i b_{\mathrm{do}}
-\tfrac12\Delta r\bigr)(s,\sigma)
\,\mathrm{d}s\biggr)\xi(\sigma),
\end{split}
\end{align*}
Note that $b_{\mathrm{do}}$ is the first order approximation of $b$ as
$r\to\infty$. By mimicking the proof of Corollary \ref{cor:strongbb}
we obtain the following result.

\begin{corollary}
\label{cor:dollard}
In addition to Condition~\ref{cond:12.6.2.21.13c},
suppose that $q$ is of Dollard type, 
and that \eqref{eq:1609189} holds.
\begin{enumerate}
\item
For any $\lambda\in\mathcal I$ and $\psi\in B$
there exist $F^\pm_{\mathrm{do}}(\lambda)\psi\in \mathcal G$ such that 
\begin{align*}
R(\lambda\pm\mathrm i0)\psi-\phi^\pm_{\mathrm{do}}
[F^\pm_{\mathrm{do}}(\lambda)\psi]\in B^*_0.
\end{align*}
In addition, $F^\pm_{\mathrm{do}}(\lambda)\psi$ 
are separately continuous in $\lambda\in \mathcal I$
and $\psi\in \mathcal G$.
\item
The operators 
\begin{align*}
F^\pm_{\mathrm{do}}=\int_{\mathcal I} \oplus F^\pm_{\mathrm{do}}(\lambda)\,\d \lambda
\end{align*} 
considered as mappings $B\cap\mathcal
H_{\mathcal I}\to \widetilde{\mathcal H}_{\mathcal I}$ 
extend uniquely
to  isometric  operators
$\mathcal H_{\mathcal I}\to \widetilde{\mathcal H}_{\mathcal I}$. 
\item
The above $F^\pm_{\mathrm{do}}$ are  related to $F^\pm$
as follows:  
\begin{align*}
    F_{\rm do}^\pm=\e^{\mp\i \theta}F^\pm;\quad
    \theta(\lambda,\sigma) = \int^\infty_{r_0}(b_{\rm
      do}-b)(s,\sigma)\,\mathrm{d}s.
     \end{align*}
\end{enumerate}
\end{corollary}

Next we set for $h\in\widetilde{\mathcal H}_{\mathcal I}$
 \begin{align}
\begin{split}
 U_{\rm do}^\pm(t)h(r,\sigma)
=&(2\pi)^{-1/2}{\mathrm e^{\mp 3\mathrm i\pi/4}}
1_{[r_0,\infty)}(r)
\mathrm e^{\pm\mathrm i K_{\rm do}}\e^{-\int_{r_0}^r(\Delta r)(s,{}\cdot{})/2\,\mathrm ds}
\\&\phantom{{}={}}{}
\cdot\bigl(\tfrac{r-r_0}{t^2}\bigr)^{1/2}
h\bigl(\tfrac{(r-r_0)^2}{2t^2}+\lambda_0(\sigma),\sigma\bigr),
\end{split}
\label{17053120}
\end{align}
where
\begin{align*}
K_{\rm do}=\tfrac{(r-r_0)^2}{2t}-t\lambda_0(\sigma)-\tfrac
  t{(r-r_0)}\int^r_{r_0}\bigl(q_1(s,\cdot)-\lambda_0(\sigma)\bigr)\, \d s.
\end{align*} As for the short-range type dynamics the Dollard type
dynamics are  asymptotically
 isometric families  of contractions. 
\begin{thm}\label{thm:14.3.12.18.20dol}
Suppose Condition~\ref{cond:12.6.2.21.13bb},  $q$ is of Dollard type, 
and that \eqref{eq:1609189} holds.
For all  $h\in \widetilde{\mathcal H}_{\mathcal I}$
there exist the limits 
\begin{align}
  \label{eq:53dol}
  W_{\rm do}^\pm h:=\lim_{t\to \infty}\e^{\pm\i tH}  U_{\rm
    do}^\pm(t)h\quad \mbox{in } \mathcal H.
\end{align} 
These limits $W_{\rm do}^\pm$ are  unitary  operators 
$W_{\rm do}^\pm\colon \widetilde{\mathcal H}_{\mathcal I}\to\mathcal H_{\mathcal I}$,
and  
  \begin{align}\label{eq:54dol}
  F_{\rm do}^\pm=(W_{\rm do}^\pm)^*,
\end{align}
respectively. Whence  the wave operators $W^\pm_{\mathrm{do}}$ 
  are complete on $\mathcal I$.
\end{thm}

Similarly to the previous subsection, 
the standard hyperbolic space and the Euclidean space fit also into this framework 
with a positive constant $\lambda_0(\sigma)\equiv \lambda_0>0$
and the zero constant $\lambda_0(\sigma)\equiv 0$, respectively.

\section{Proofs}\label{subsec:Time-dependent
wave operator}

\subsection{Leading asymptotics of comparison dynamics}
Here we study basic properties of 
the comparison dynamics 
$U^\pm(t)$ defined by \eqref{eq:14.3.12.20.22}.
The main result of this subsection is Lemma~\ref{lem:14.4.29.22.28},
which extracts the leading asymptotics of $U^\pm(t)h$ as $t\to\infty$.
Actually, the integrand of \eqref{eq:14.3.12.20.22} has an oscillatory factor with phase
\begin{align}
\Theta_\pm=\pm\Theta;\quad 
\Theta(\lambda,t,r,\sigma)
=\int^r_{r_0} \tilde b_\lambda(s,\sigma) \,\d s-t\lambda,
\label{eq:160920}
\end{align}
and we are going to employ the 
one-dimensional stationary phase argument; see \cite{II} 
for a somewhat related  analysis in higher dimensions. Throughout  this
subsection we  assume Condition~\ref{cond:12.6.2.21.13c}, 
and all  proofs  are given only for
the upper sign, since this is sufficient due to time reversal 
invariance.  

For 
technical reasons it turns out to be appropriate  to use  the
stationary phase method with a modified phase
depending on $h$. We will 
choose a constant $r_1\ge r_0$ depending 
on $\mathop{\mathrm{supp}}h$ 
and consider a stationary point of the function
\begin{align}
\Theta_1(\lambda,t,r,\sigma)
=\int^r_{r_1} \tilde b_\lambda(s,\sigma) \,\d s-t\lambda,
\label{eq:160920b}
\end{align}
instead of the function $\Theta$ of \eqref{eq:160920}.
We  study  such  stationary point in the following lemma in which $h$
at most enters in a disguised form:
The parameters $\lambda_1>\lambda_0$ and $D\subseteq S$ 
in the lemma 
  will later be chosen  by the requirement 
  $\mathop{\mathrm{supp}}h\subseteq (\lambda_1,\infty)\times D$ after having fixed  $h\in C^1_{\mathrm
    c}(\mathcal I\times S)$.

\begin{lemma}\label{lem:14.4.29.22.28bb} 
Let $\lambda_1>\lambda_0$  and 
$D\subseteq S$ be a relatively compact open subset. 
Fix $r_1\ge r_0$  such that
\begin{subequations}
\begin{align}
[r_1,\infty)\times \overline D\subseteq E,\quad 
r_1\geq  r_{\lambda_1}=\sup_{\lambda\ge \lambda_1}r_\lambda,
\label{eq:16092114}
\end{align} 
where $r_\lambda$ is defined in agreement with \eqref{eq:14.6.29.23.46}.
With $\Theta_1$ given by \eqref{eq:160920b} in terms of this $r_1$ we set  
\begin{align}
\Omega_c=\bigl\{(t,r,\sigma)\in (0,\infty)\times (r_1, \infty)\times D\,\big|\,
\partial_\lambda\Theta_1(\lambda_1,t,r,\sigma)
>0\bigr\}.
\label{eq:17060819}
\end{align} 
\end{subequations}

\begin{subequations}
For any $(t,r,\sigma)\in \Omega_c$
there exists  a unique  solution $\lambda_c >\lambda_1$ to the equation
\begin{align}\label{eq:critical}
(\partial_\lambda\Theta_1)(\lambda_c,t,r,\sigma)=0.
\end{align} 
This  solution $\lambda_c=\lambda_c(t,r,\sigma)$ 
is $C^1$  (more generally $C^k$ if $q_1$ is $C^k$) 
and satisfies 
\begin{align}
\partial_t\lambda_c
+b_{\lambda_c}\partial_r\lambda_c=0,\quad
  \partial_r\lambda_c>0.
\label{eq:derilamdab}
\end{align}
In addition, 
there exist constants $c,C>0$ such that 
for any $(t, r,\sigma)\in \Omega_c$
\begin{align}
  \label{eq:twosidedest}
  c(r-r_1)^2/t^2\le\lambda_c(t,r,\sigma)-\lambda_0\le C(r-r_1)^2/t^2.
\end{align} 
 \end{subequations}

\begin{subequations}
Moreover, letting 
\begin{align}
K_1(t, r,\sigma)&=\Theta_1(\lambda_c(t,r,\sigma),t,r,\sigma)
\quad \text{for }(t, r,\sigma)\in \Omega_c,
\label{eq:16092123}
\end{align} 
  the following identities hold:
\begin{align}
  \label{eq:Kder}
  \partial_tK_1=-\lambda_c, \quad \partial_r K_1=\tilde b_{\lambda_c}.
\end{align}
In particular   $K_1$ is a solution to the Hamilton--Jacobi equation
\begin{align}\label{eq:hamilJacobi_equ}
\partial_t K_1+\tfrac 12(|\mathrm dr|\partial_r K_1)^2+q_1=0.
\end{align} 
\end{subequations}
\end{lemma}
\begin{proof}
We first note that thanks to \eqref{eq:16092114}
we have the following  expressions when $\lambda\geq \lambda_1$ and
$(r,\sigma)\in [r_1,\infty)\times \overline D$:
\begin{subequations}
\begin{align}
\Theta_1(\lambda,t,r,\sigma)
&=
\int^r_{r_1} |\mathrm dr|^{-1}[2(\lambda-q_1)]^{1/2}(s,\sigma) \,\d s-t\lambda,
\label{eq:1609212344}
\\
\partial_\lambda\Theta_1(\lambda,t,r,\sigma)
&=
\int^r_{r_1} |\mathrm dr|^{-1}[2(\lambda-q_1)]^{-1/2}(s,\sigma) \,\d s-t,
\label{eq:1609211456}\\
\partial_\lambda^2\Theta_1(\lambda,t,r,\sigma)
&=-\int^r_{r_1}  |\mathrm dr|^{-1} [2(\lambda-q_1)]^{-3/2} \,\d s.\label{eq:1609211456c}
\end{align} 
\end{subequations}
For any fixed $(t,r,\sigma)\in \Omega_c$ 
the quantity
\eqref{eq:1609211456} is positive for $\lambda=\lambda_1$, 
monotonically decreasing in $\lambda>\lambda_1$,
and takes negative values for large $\lambda>\lambda_1$.
Hence there exists a unique solution $\lambda_c=\lambda_c(t,r,\sigma)$ 
to \eqref{eq:critical}, and  by
the implicit function theorem it is $C^1$.

Noting the expression \eqref{eq:1609211456}, 
we can differentiate \eqref{eq:critical}, and 
obtain formulas
\begin{align}\label{eq:derilamda}
\begin{split}
\partial_t\lambda_c&=-\biggl(\int^r_{r_1} 
|\mathrm dr|^{-1}[2(\lambda_c-q_1)]^{-3/2} \,\d s\biggr)^{-1},
\\
  \partial_r\lambda_c&
=|\mathrm dr|^{-1}[2(\lambda_c-q_1)]^{-1/2} 
\biggl(\int^r_{r_1} 
|\mathrm dr|^{-1}[2(\lambda_c-q_1)]^{-3/2}\,\d s\biggr)^{-1},
\end{split}
\end{align} 
which verifies \eqref{eq:derilamdab}.
The bounds in \eqref{eq:twosidedest} are verified 
easily by inserting $\lambda=\lambda_c$ in \eqref{eq:1609211456}
(taken equal to zero)
and estimating the integral using \eqref{eq:14.6.29.23.46}. (In fact by
this argument the upper
bound of \eqref{eq:twosidedest} holds with $C=1$.)

We can verify the formulas in \eqref{eq:Kder}
by differentiating the definition \eqref{eq:16092123}
and using \eqref{eq:critical} and \eqref{eq:1609212344}.
Obviously  \eqref{eq:hamilJacobi_equ} is a  consequence 
of \eqref{eq:Kder}.
\end{proof}

We can now show a basic property of $U^\pm(t)h$ defined by \eqref{eq:14.3.12.20.22}.

\begin{lemma}\label{lem:14.4.29.22.28b} 
For each $t\geq 0$ and  $h\in C^1_{\mathrm c}(\mathcal I\times S)$
the functions $U^\pm(t)h$ belong to $\mathcal H$. Moreover
$U^\pm(\cdot)h$ is a continuous $\mathcal H$-valued function.
\end{lemma}
\begin{proof}
Let $h\in C^1_{\mathrm c}(\mathcal I\times S)$ and $T>0$ be given. It
suffices to show that the function $[0,T]\ni t\to  U^+(t)h\in \mathcal H$ is
well-defined and 
continuous.

We pick   a  number $\lambda_1>\lambda_0$ and a relatively compact open subset
$D\subseteq S$ such that
\begin{align}
\mathop{\mathrm{supp}}h\subseteq (\lambda_1,\infty)\times D.
\label{eq:17060820}
\end{align}
We fix $r_1\ge r_0$ satisfying
\eqref{eq:16092114} in agreement  with Lemma~\ref{lem:14.4.29.22.28bb} with input given by  
 the above $\lambda_1$ and $D$. 

Let us write for short, using  \eqref{eq:160920b},  
\begin{align}
U^+(t)h(r,\sigma)
=
\int_{\mathcal I} 
\e^{\i \Theta_1(\lambda,t,r,\sigma)}
\Phi(\lambda,r,\sigma)
\,\d \lambda,
\label{eq:1609251aa}
\end{align}
where 
\begin{align*}
\Phi(\lambda,r,\sigma)
=(2\pi\i)^{-1}
\eta_\lambda(r)[2|\mathrm dr|^2 (\lambda-q_1)]^{-1/4}
\mathrm e^{-\int_{r_0}^{r}\mathop{\mathrm{div}}\tilde\omega(s,\sigma)/2\,\mathrm{d}s}
\parb{\e^{\i (\Theta-\Theta_1)}h}(\lambda,\sigma).
\end{align*}
Note that indeed $\Theta-\Theta_1$ is  independent of $(t,r)$,
and we omit these variables:
\begin{align*}
(\Theta-\Theta_1)(\lambda,\sigma)
&=
\int^{r_1}_{r_0} \tilde b_\lambda(s,\sigma) \,\d s.
\end{align*} 
We also note that $\Phi(\lambda,{}\cdot{})$ belongs to $B^*$ uniformly 
in $\lambda\in\mathcal I$.
With an extra decay factor $r^{-\delta}$,
$\delta>1/2$, obviously we obtain a vector in $\mathcal H$. In
particular  $\chi_n U^+(t)h$ for any $n$, and this leads
us to consider $\bar\chi_n U^+(t)h$ only. We choose and fix 
  $n$  so large that $R_n=2^n>r_1$  and that for  some $c,C>0$  the bounds
\begin{align*}
\partial_\lambda\Theta_1\geq c (|\lambda|+1)^{-1/2}\,r,\quad 
|\partial_\lambda^2\Theta_1|\le Cr
\end{align*} hold  for $\lambda\geq \lambda_1$, $t\leq
T$, $r\geq R_n$ and
$\sigma\in \overline D$. These bounds are immediate from  the formulas
\eqref{eq:1609211456} and \eqref{eq:1609211456c}.
 We insert $\e^{\i \Theta_1}
=   (\i \partial_\lambda \Theta_1)^{-1}\partial_\lambda\e^{\i \Theta_1}$ in \eqref{eq:1609251aa}
   and do  a single
 integration    by parts, which is  legitimate since
 $\e^{\i(\Theta-\Theta_1)}h\in C^1$. The  bounds  provide  an extra decay factor $r^{-1}$
for $\bar\chi_n U^+(t)h$.
Hence we have shown that $ U^+(t)h\in \mathcal H$ for $t\leq T$.  

The
continuity  statement
follows from the resulting representation of $U^\pm(t)h$ after doing the  
integration by parts; we omit the details.
\end{proof}

Now  letting  $h\in C^1_{\mathrm c}(\mathcal I\times S)$ be given we aim at 
extracting  the  leading term of $U^\pm(t)h$ as $t\to\infty$. 
As in the proof of Lemma \ref{lem:14.4.29.22.28b} we can take a 
number $\lambda_1>\lambda_0$ and a relatively compact open subset
$D\subseteq S$ such that \eqref{eq:17060820} is fulfilled. Again we fix $r_1\ge r_0$ satisfying
\eqref{eq:16092114} in agreement  with Lemma~\ref{lem:14.4.29.22.28bb} with input given by  
 the above $\lambda_1$ and $D$. 
 Let $\lambda_c=\lambda_c(t,r,\sigma)$ 
be the  solution to \eqref{eq:critical}.
Then we set
\begin{align*}
 \Omega_c(t)&=\{(r,\sigma)|\, (t,r,\sigma)\in\Omega_c\};\, t>0, 
\end{align*}
and 
\begin{align}
\begin{split}
U^\pm_0(t)h(r,\sigma)&
=(2\pi)^{-1/2}\mathrm e^{\mp3\mathrm i\pi/4}
1_{\Omega_c(t)}(r,\sigma)
\mathrm e^{\pm\i K(t, r,\sigma)}
\mathrm e^{-\int_{r_0}^{r}(\mathop{\mathrm{div}}\tilde\omega)(s,\sigma)/2\,\mathrm{d}s}
\\&\phantom{{}={}}{}\cdot
(\partial_r\lambda_c(t,r,\sigma))^{1/2}
h(\lambda_c(t,r,\sigma),\sigma),
\end{split}
\label{eq:1609220}
\end{align}
where
\begin{align*}
K(t, r,\sigma)=\Theta(\lambda_c(t,r,\sigma),t,r,\sigma).
\end{align*} 
The
factor $1_{\Omega_c(t)}$ is essentially redundant, 
since the support of the factor
$h(\lambda_c(t,\cdot),\cdot)$ is contained in $\Omega_c(t)$ which in turn is
an easy consequence of \eqref{eq:twosidedest}. In fact it is not
difficult to show using \eqref{eq:twosidedest} and \eqref{eq:derilamda}
 that $U^\pm_0(t)h\in
C^1_{\mathrm c}(M)$ for any $t>0$. 
We also note that  the right-hand side of \eqref{eq:1609220}
depends on a choice of parameters $\lambda_1$, $D$ and $r_1$ which  possibly
have  a 
non-linear dependence of  $h$.
Hence  the  operator-like notation $U^\pm_0(t)$ is somewhat  abuse of
notation, however we prefer to use it for simplicity.

\begin{lemma}\label{lem:14.4.29.22.28} 
Under the above assumptions 
$U^\pm_0(\cdot)h$ are 
continuously differentiable $\vH$-valued functions
in $t>0$,
and satisfy that for all $t>0$  
\begin{align}
\|U^\pm_0(t)h\|_{\mathcal H}=\|h\|_{\widetilde{\mathcal H}_{\mathcal I}},
\label{eq:16092213}
\end{align}
and 
\begin{align}
  \begin{split}
\tfrac{\mathrm d}{\mathrm dt}U^\pm_0(t)h
&=-\mathrm iG^\pm(t)U^\pm_0(t)h;\quad
G^\pm(t)
=
{\mathop{\mathrm{Re}}\bigl(\tilde b_{\lambda_c}A\bigr)}
\mp\tfrac12|\d r|^2\tilde b_{\lambda_c}^2
\pm q_1,
\label{eq:14.5.5.12.29B} 
  \end{split}
\end{align}
respectively. 
Moreover, 
\begin{align}
U^\pm(t)h=U^\pm_0(t)h
+O_{\mathcal H}(t^{-1/8}) \quad\text{as }t\to \infty.
\label{eq:14.5.9.12.13BB}
\end{align} 
\end{lemma}
\begin{proof} 
 
\textit{Step I.}\quad Since
  $\partial_r\partial_\lambda\Theta_1(\lambda_1,t,r,\sigma) =
  b_{\lambda_1}^{-1}>0$ on $\Omega_c$ we see that $\lambda_c$ as a function
  of $r$ only is defined on a half-axis, say
  $(r_c(t,\sigma),\infty)$. In fact we know from \eqref{eq:derilamda} that
  $\lambda_c(t,\cdot ,\sigma)$ is increasing.  Thanks to
  \eqref{eq:twosidedest} this  function  tends to $\infty$ 
 {as} $r\to
  \infty$.  The left end point $r=r_c(t,\sigma)$ is the biggest solution to
  the equation
  \begin{align*}
    \partial_\lambda\Theta_1(\lambda_1,t,r,\sigma)=0,
  \end{align*}  and it is easy to see
  that
  \begin{align*}
    \lambda_c(t,r,\sigma)\to \lambda_1\text{ for }r\searrow r_c(t,\sigma).
  \end{align*}

  Due to these remarks the norm identity \eqref{eq:16092213} follows
  by first writing the square of the left-hand side as an integral in
  the spherical coordinates $(r,\sigma)$ and then changing to the
  variables $(\lambda,\sigma)=(\lambda_c(t,r,\sigma),\sigma)$.
  Obviously this also verifies that $U^+_0(\cdot)h$ is $\mathcal
  H$-valued.

\smallskip
\noindent
\textit{Step I\hspace{-.1em}I.}\quad 
{Next, we show \eqref{eq:14.5.5.12.29B},
which in particular implies the continuous differentiability of 
$U_0^+(\cdot)h$.}
To compute the derivative as \eqref{eq:14.5.5.12.29B}
let us write \eqref{eq:1609220} as 
\begin{align}
\begin{split}
U^+_0(\cdot)h&
=
(2\pi)^{-1/2}\mathrm e^{-3\mathrm i\pi/4}
1_{\Omega_c(\cdot)}
\mathrm e^{\i K_1}
\mathrm e^{-\int_{r_0}^{r}(\mathop{\mathrm{div}}\tilde\omega)/2\,\mathrm{d}s}
(\partial_r\lambda_c)^{1/2}
(\mathrm e^{\i (\Theta-\Theta_1)}h)(\lambda_c,\sigma).
\end{split}
\label{eq:1609220b}
\end{align} 
We differentiate \eqref{eq:1609220b} in $t>0$.
 The factor $1_{\Omega_c}$ can be ignored, cf. the a discussion above
 the lemma.  Note also that due to \eqref{eq:derilamda} the quantities
$\partial_t\lambda_c$ and $\partial_r\lambda_c$ are $C^1$. 
By \eqref{eq:Kder}  and \eqref{eq:hamilJacobi_equ} we have 
\begin{align}
\begin{split}
\partial_t\mathrm e^{\i K_1}
&=-\mathrm i\lambda_c\mathrm e^{\i K_1}
\\&
=-\mathrm i\big(\tfrac12 |\mathrm dr|^2\tilde b_{\lambda_c}^2
+q_1\bigl)\mathrm e^{\i K_1}
\\&
=-\mathrm i\big(\tilde b_{\lambda_c}p^r
-\tfrac12 |\mathrm dr|^2\tilde b_{\lambda_c}^2+q_1\bigl)\mathrm e^{\i K_1}
.
\end{split}\label{eq:16092325}
\end{align}
Next, using \eqref{eq:derilamdab}, we can compute 
\begin{align*}
\partial_t(\partial_r\lambda_c)^{1/2}
&
=\tfrac12 (\partial_r\partial_t\lambda_c)(\partial_r\lambda_c)^{-1/2}
\\&
=-\tfrac12 
\bigl(\partial_r|\mathrm dr|^2\tilde b_{\lambda_c}\partial_r\lambda_c\bigr)
(\partial_r\lambda_c)^{-1/2}
\\&
=
-|\mathrm dr|^2\tilde b_{\lambda_c}
\partial_r(\partial_r\lambda_c)^{1/2}
-\tfrac12 \bigl(|\mathrm dr|^{-2}\partial^r|\mathrm dr|^2\tilde b_{\lambda_c}\bigr)
(\partial_r\lambda_c)^{1/2}
\\&
=
-\mathrm i\tilde b_{\lambda_c}p^r
(\partial_r\lambda_c)^{1/2}
-\tfrac12 \Bigl[
\mathop{\mathrm{div}}\bigl(|\mathrm dr|^2\tilde b_{\lambda_c}\tilde \omega\bigr)
-|\mathrm dr|^2\tilde b_{\lambda_c}(\mathop{\mathrm{div}}\tilde \omega)
\Bigr]
(\partial_r\lambda_c)^{1/2}
\\&
=
-\mathrm i\tilde b_{\lambda_c}p^r
(\partial_r\lambda_c)^{1/2}
-\mathrm i\tfrac12 
\bigl((p^r)^*\tilde b_{\lambda_c}\bigr)(\partial_r\lambda_c)^{1/2}
+\tfrac12 |\mathrm dr|^2\tilde b_{\lambda_c}(\mathop{\mathrm{div}}\tilde \omega)
(\partial_r\lambda_c)^{1/2}
,
\end{align*}
so that 
\begin{align}
\begin{split}
\partial_t
\mathrm e^{-\int_{r_0}^{r}(\mathop{\mathrm{div}}\tilde\omega)/2\,\mathrm{d}s}
(\partial_r\lambda_c)^{1/2}
&
=
-\mathrm i\bigl[\tilde b_{\lambda_c}p^r+\tfrac12 \bigl((p^r)^*\tilde b_{\lambda_c}\bigr)\bigr]
(\partial_r\lambda_c)^{1/2}
\mathrm e^{-\int_{r_0}^{r}(\mathop{\mathrm{div}}\tilde\omega)/2\,\mathrm{d}s}.
\end{split}\label{eq:16092326}
\end{align}
Finally by \eqref{eq:derilamdab} again we have 
\begin{align}
\begin{split}
\partial_t(\mathrm e^{\i (\Theta-\Theta_1)}h)(\lambda_c,\sigma)
&
=(\partial_t\lambda_c)(\partial_\lambda\mathrm e^{\i (\Theta-\Theta_1)}h)(\lambda_c,\sigma)
\\&
=-(\tilde b_{\lambda_c}\partial^r\lambda_c)
(\partial_\lambda\mathrm e^{\i (\Theta-\Theta_1)}h)(\lambda_c,\sigma)
\\&
=-\mathrm i\tilde b_{\lambda_c}p^r
(\mathrm e^{\i (\Theta-\Theta_1)}h)(\lambda_c,\sigma).
\end{split}\label{eq:16092327}
\end{align}
Using  \eqref{eq:16092325}, \eqref{eq:16092326}, \eqref{eq:16092327}
and the product rule 
{we obtain 
\begin{align*}
G^+
=
\mathop{\mathrm{Re}}\bigl(\tilde b_{\lambda_c}p^r\bigr)
-\tfrac12|\d r|^2\tilde b_{\lambda_c}^2
+ q_1,
\end{align*}
and hence \eqref{eq:14.5.5.12.29B} follows.}

\smallskip
\noindent
\textit{Step I\hspace{-.1em}I\hspace{-.1em}I.}\quad
It remains to show \eqref{eq:14.5.9.12.13BB}. 
We are going to apply the stationary phase method, and we start by
doing  some
cut-offs.  
Using again  \eqref{eq:1609251aa} we note (as before) that $\Phi(\lambda,{}\cdot{})$ belongs to $B^*$ uniformly 
in $\lambda\in\mathcal I$ and that 
with  an extra decay factor $r^{-1/2-\epsilon}$,
$\epsilon>0$,  we obtain a vector in $\mathcal H$.
For any $M>m>0$ we can  split the integral  as 
\begin{align}
\begin{split}
U^+(t)h
&=
1_{(m,M)}(r/t)
{1_{\Omega_c(t)}}(\cdot)
\int_{\mathcal I} 
\e^{\i \Theta_1(\lambda,t,\cdot)}
\Phi(\lambda,\cdot)
\,\d \lambda
\\&\phantom{{}={}}{}
+1_{(0,m]}(r/t)1_{(r_1,\infty)}(r)\int_{\mathcal I} 
\e^{\i \Theta_1(\lambda,t,\cdot)}
\Phi(\lambda,\cdot)
\,\d \lambda
\\&\phantom{{}={}}{}
+1_{(0,m]}(r/t)1_{(r_0,r_1]}(r)\int_{\mathcal I} 
\e^{\i \Theta_1(\lambda,t,\cdot)}
\Phi(\lambda,\cdot)
\,\d \lambda
\\&\phantom{{}={}}{}
+1_{[M,\infty)}(r/t)\int_{\mathcal I} 
\e^{\i \Theta_1(\lambda,t,\cdot)}
\Phi(\lambda,\cdot)
\,\d \lambda
\\&\phantom{{}={}}{}
+1_{(m,M)}(r/t)1_{\Omega_c(t)^c}(\cdot)
\int_{\mathcal I} 
\e^{\i \Theta_1(\lambda,t,\cdot)}
\Phi(\lambda,\cdot)
\,\d \lambda.
\end{split}
\label{eq:160925143}
\end{align}
Using the expression \eqref{eq:1609211456}  we can pick $m>0$ small enough
such that for some $c_1>0$ the following bound holds for all large $t$
uniformly in
$r>r_1$, $r/t\le m$ and $(\lambda,\sigma)\in\mathop{\mathrm{supp}}h$:
\begin{align}
\partial_\lambda\Theta_1(\lambda,t,r,\sigma)
\le -c_1t.
\label{eq:160925223}
\end{align}
With \eqref{eq:160925223}  and \eqref{eq:1609211456c}
we can treat the second term on the right-hand side of \eqref{eq:160925143}
by an integration parts, yielding that 
\begin{align*}
\left\|1_{(0,m]}(r/t)1_{(r_1,\infty)}(r)\int_{\mathcal I} 
\e^{\i \Theta_1(\lambda,t,\cdot)}
\Phi(\lambda,\cdot)
\,\d \lambda\right\|_{\mathcal H}
\le C_{1} t^{-1/2}.
\end{align*} The third term is treated similarly  by using the phase
function $-\lambda t$ instead of $\Theta_1$, yielding the same  bound.

Similarly   by  taking $M>0$  large enough (this part is very similar to the
proof of Lemma \ref{lem:14.4.29.22.28b}), we can bound the fourth term
of \eqref{eq:160925143} as
\begin{align*}
\left\|1_{[M,\infty)}(r/t)\int_{\mathcal I} 
\e^{\i \Theta_1(\lambda,t,\cdot)}
\Phi(\lambda,\cdot)
\,\d \lambda\right\|_{\mathcal H}
\le C_{2} t^{-1/2}.
\end{align*}
Next, let us consider the fifth  term. 
By \eqref{eq:17060819}, \eqref{eq:1609211456c} and 
\eqref{eq:17060820} it follows that 
for large $t>0$
and on $\mathop{\mathrm{supp}}\parb{1_{(m,M)}(r/t)1_{\Omega_c(t)^c}(\cdot)\Phi(\cdot)}$
\begin{align*}
\partial_\lambda\Theta_1(\lambda,t,r,\sigma)
&\le 
\partial_\lambda\Theta_1(\lambda,t,r,\sigma)
-\partial_\lambda\Theta_1(\lambda_1,t, r,\sigma)
\\&
\le -c_2(r-r_1)
\\&
\le -c_3t.
\end{align*}
Whence by an integration by parts we obtain the  bound $O(t^{-1/2})$ again:
\begin{align*}
\left\|1_{(m,M)}(r/t)
{1_{\Omega_c(t)^c}(\cdot)}
\int_{\mathcal I} 
\e^{\i \Theta_1(\lambda,t,{}\cdot{})}
\Phi(\lambda,{}\cdot{})
\,\d \lambda\right\|_{\mathcal H}
\le C_3 t^{-1/2}.
\end{align*}

\smallskip
\noindent
\textit{Step I\hspace{-.1em}V.}\quad
It remains  to compare  the first term of \eqref{eq:160925143} with \eqref{eq:1609220b}.
Let us look at the phase in \eqref{eq:160925143} in   more detail.
By a Taylor expansion we have, assuming the front factor $1_{(m,M)}(r/t)1_{\Omega_c(t)}(r,\sigma)=1$,
\begin{align*}
\Theta_1(\lambda,t,x)
&=K_1(t,x)
+\tfrac12(\partial_\lambda^2\Theta_1)
(\lambda_c,x)(\lambda-\lambda_c)^2
+\Xi(\lambda,\lambda_c,x)(\lambda-\lambda_c)^3
\end{align*}
with
\begin{align*}
\Xi(\lambda,\lambda_c,x)
&=\tfrac12\int_0^1(1-\kappa)^2(\partial_\lambda^3\Theta_1)
(\lambda_c+\kappa(\lambda-\lambda_c),x)\,\mathrm
d \kappa.
\end{align*}
Here, since $\partial_\lambda^k\Theta_1$ is independent of $t$ for $k=2,3,\ldots$,
we have omitted the $t$ variable for short.
Then we further decompose the first term of \eqref{eq:160925143}
(assuming still  $1_{(m,M)}(r/t)1_{\Omega_c(t)}(r,\sigma)=1$)  as 
\begin{align}
\begin{split}
&\int_{\mathcal I} 
\e^{\i \Theta_1(\lambda,t,{}\cdot{})}
\Phi(\lambda,{}\cdot{})
\,\d \lambda
\\&
=
\mathrm e^{\mathrm iK_1(t,{}\cdot{})}
\Phi(\lambda_c,{}\cdot{})\int_{\mathbb R} 
\e^{\i(\partial_\lambda^2\Theta_1)(\lambda_c,{}\cdot{})(\lambda-\lambda_c)^2/2}
\,\d \lambda
\\&\phantom{{}={}}{}
-\mathrm e^{\mathrm iK_1(t,{}\cdot{})}
\Phi(\lambda_c,{}\cdot{})\int_{\mathbb R\setminus I_c} 
\e^{\i(\partial_\lambda^2\Theta_1)(\lambda_c,{}\cdot{})(\lambda-\lambda_c)^2/2}
\,\d \lambda
\\&\phantom{{}={}}{}
+
\int_{\mathcal I\setminus I_c} 
\e^{\i \Theta_1(\lambda,{}\cdot{})}\Phi(\lambda,\cdot)
\,\d \lambda
\\&\phantom{{}={}}{}
+\mathrm e^{\mathrm i K_1(t,\cdot)}
\int_{I_c} 
\e^{\i(\partial_\lambda^2\Theta_1)(\lambda_c,{}\cdot{})(\lambda-\lambda_c)^2/2}
\Bigl[\mathrm e^{\mathrm i\Xi(\lambda,\cdot)(\lambda-\lambda_c)^3}
\Phi(\lambda,\cdot)
-\Phi(\lambda_c,\cdot)\Bigr]
\,\d \lambda
\\&
=:J_1(t,\cdot)+J_2(t,\cdot)
+J_3(t,\cdot)+J_4(t,\cdot),
\end{split}
\label{16092512}
\end{align}
where $I_c=I_c(t,x)
=[\lambda_c(t,x)-t^{-\delta},\lambda_c(t,x)+t^{-\delta}]$, $\delta\in (0,1/2)$.
Since by \eqref{eq:1609211456c}  and
\eqref{eq:derilamda}
\begin{align*}
(\partial_\lambda^2\Theta_1)(\lambda_c,\cdot)
=-\int^r_{r_1} |\mathrm dr|^{-1}[2(\lambda_c-q_1)]^{-3/2}(s,\sigma) \,\d s
=-(b_{\lambda_c}\partial_r\lambda_c)^{-1},
\end{align*}
we can do the Gaussian integration in the first term of
\eqref{16092512} (see for example \cite[Theorem 7.6.1]{Ho}) and obtain using
\eqref{eq:twosidedest} that
\begin{align*}
1_{(m,M)}(r/t){1_{\Omega_c(t)}(\cdot)}J_1(t,\cdot)=U^+_0(t)h.
\end{align*}
Hence we are left with bounding the three other   terms of \eqref{16092512}.
To bound the second term of \eqref{16092512}
we use an integration by parts, noting 
\begin{align*}
\big|(\partial_\lambda^2\Theta_1)(\lambda_c,\cdot)(\lambda-\lambda_c)\bigr|
\ge c_{\delta,1}rt^{-\delta}
.
\end{align*}
This yields an extra decay factor $t^{-1+\delta}$ 
along with the cut-off $1_{(m,M)}(r/t)$.
We also recall that $\Phi(\lambda,\cdot)$ belongs to $B^*$ uniformly 
in $\lambda\in\mathcal I$, so that with $1_{(m,M)}(r/t)$ 
\begin{align*}
\bigl\|1_{(m,M)}(r/t){1_{\Omega_c(t)}(\cdot)}
\mathrm e^{\mathrm iK_1(t,\cdot)}
\Phi(\lambda_c,\cdot)\bigr\|_{\mathcal H}
\le C_2t^{1/2}.
\end{align*} 
{Hence 
\begin{align}\label{eq:firT}
  \|1_{(m,M)}(r/t)1_{\Omega_c(t)}(\cdot)J_2(t,\cdot)\|_{\mathcal H}
  \le C_{\delta,1}t^{-1/2+\delta}.
\end{align}
To bound the third term of \eqref{16092512} 
let us implement yet another  integration by parts:
\begin{align}
\begin{split}
J_3(t,\cdot)
&=
\bigl[(\mathrm i\partial_\lambda\Theta_1)^{-1}\e^{\i \Theta_1}\Phi
\bigr]_{|\lambda=\lambda_c-t^{-\delta}}
-\bigl[(\mathrm i\partial_\lambda\Theta_1)^{-1}\e^{\i \Theta_1}\Phi
\bigr]_{|\lambda=\lambda_c+t^{-\delta}}
\\&\phantom{{}={}}{}
-\int_{\mathcal I\setminus I_c} 
\e^{\i \Theta_1(\lambda,{}\cdot{})}
\partial_\lambda\bigl[
(\mathrm i\partial_\lambda\Theta_1)^{-1}\Phi(\lambda,{}\cdot{})\bigr]
\,\d \lambda.
\end{split}\label{17060822}
\end{align}
Here,  including  the cut-off $1_{(m,M)}(r/t)1_{\Omega_c(t)}(\cdot)$
and taking note of  
the integration region $\mathcal I\setminus I_c$, 
we have 
\begin{align*}
|\partial_\lambda\Theta_1|
\ge c_{\delta,2}rt^{^{-\delta}}\ge c_{\delta,3}t^{1-\delta},
\end{align*}
so that the contributions from the boundary terms in \eqref{17060822} 
are estimated as
\begin{align*}
\Bigl\|1_{(m,M)}(r/t)1_{\Omega_c(t)}(\cdot)
\bigl[(\mathrm i\partial_\lambda\Theta_1)^{-1}\e^{\i \Theta_1}\Phi
\bigr]_{|\lambda=\lambda_c\pm t^{-\delta}}
\Bigr\|_{\mathcal H}
\le C_{\delta,2}t^{-1/2+\delta}.
\end{align*}
Moreover, by the product rule,   a part of the integral in \eqref{17060822} estimated as 
\begin{align*}
\Biggl\|1_{(m,M)}(r/t)1_{\Omega_c(t)}(\cdot)
\int_{\mathcal I\setminus I_c} 
\e^{\i \Theta_1(\lambda,\cdot)}
(\mathrm i\partial_\lambda\Theta_1)^{-1}\partial_\lambda\Phi(\lambda,\cdot)
\,\d \lambda
\Biggr\|_{\mathcal H}
\le C_{\delta,3}t^{-1/2+\delta}.
\end{align*} 
 The other contribution in 
\eqref{17060822} comes from differentiating the factor $(\mathrm
i\partial_\lambda\Theta_1)^{-1}$,  and we can then use
\begin{align*}
\bigl|\partial_\lambda(\partial_\lambda\Theta_1)^{-1}\bigr|
\le C_{\delta,4}t^{-1+2\delta}.
\end{align*}
Since $\delta<1/2$ the right-hand side is decaying. The bound  leads to the
estimate  $O_\vH(t^{-1/2+2\delta})$, however by repeated  integrations
by parts, we can estimate  this contribution in 
\eqref{17060822}    as  $O_\vH(t^{-1/2+\delta})$. 
To sum up we obtain
\begin{align}
\|1_{(m,M)}(r/t)1_{\Omega_c(t)}(\cdot)J_3(t,\cdot)\|_{\mathcal H}
\le C_{\delta,5} t^{-1/2+\delta}.
\label{eq:14.5.9.12.13BBAAC}
\end{align}
}Finally  for the fourth term of \eqref{16092512}  
we write
\begin{align*}
\e^{\i(\partial_\lambda^2\Theta_1)(\lambda_c,\cdot)(\lambda-\lambda_c)^2/2}
=\tfrac {\d }{\d
    \lambda}\int^\lambda_{\lambda_c}
\e^{\i(\partial_\lambda^2\Theta_1)(\lambda_c,\cdot)(\lambda'-\lambda_c)^2/2}
\,\d \lambda'
\end{align*} and perform one  integration  by parts. Using the van der
Corput Lemma, cf. \cite [p. 332]{St}, we then  obtain 
\begin{align}
\|1_{(m,M)}(r/t)1_{\Omega_c(t)}(\cdot)J_4(t,{}\cdot{})\|_{\mathcal H}
\le C_{\delta,6} t^{1-3\delta}.
\label{eq:estI}
\end{align}
  The bounds \eqref{eq:firT}, \eqref{eq:14.5.9.12.13BBAAC} and
\eqref{eq:estI} are optimized by taking
$\delta=3/8$.
Hence we obtain the desired asymptotics \eqref{eq:14.5.9.12.13BB}.
\end{proof}

\subsection{{Separation of radial and angular variables}}
In this subsection we quote results from \cite{IS3} concerning properties of the tensor $\ell$.
The statements are given in a self-contained manner; we refer to \cite[Section~2]{IS3} for  proofs.

It is clear from \eqref{eq:co_area}
that we naturally have the identification 
\begin{align*}
L^2(E)\cong L^2([r_0,\infty)_r;\mathcal G_r),\quad
 \langle\psi,\phi\rangle_{L^2(E)}
 =  \int_{r_0}^\infty\langle\psi,\phi\rangle_{\mathcal G_r}\,\mathrm dr.
\end{align*}
Such a decomposition holds
also for the Riemannian metric,
and hence for the Laplace--Beltrami operator.
\begin{lemma}\label{lem:1762113}
Suppose Condition~\ref{cond:12.6.2.21.13}.
Then in the spherical coordinates $(r,\sigma)=(r,\sigma^2,\dots,\sigma^d)$ in $E$
one has 
\begin{align*}
g=|\mathrm dr|^{-2}\,\mathrm dr\otimes \mathrm dr
+g_{\alpha\beta}\,\mathrm d\sigma^\alpha\otimes \mathrm d\sigma^\beta,
\end{align*}
where the Greek indices run over $2,\dots,d$.
In particular,  by  the definition \eqref{eq:13.9.23.5.54},
the tensor $\ell$ coincides with the spherical part of $g$:
$$\ell=g_{\alpha\beta}\,
\mathrm d\sigma^\alpha\otimes \mathrm d\sigma^\beta
\quad\text{on }E,$$
and the operator $L$ can be identified with a direct sum:
\begin{align}
\begin{split}
  L\cong 
\int^\infty_{r_0}
\oplus L_r \,\mathrm d r
\quad\text{as quadratic forms on }C^\infty_{\mathrm c}(E),
\end{split}\label{eq:15091010}
\end{align}
where $L_r$ is the Laplace--Beltrami operator on $S_r$
with respect to the induced metric $g_r:=\iota_r^*g$ 
and the (non-Riemannian) density $\mathrm d\tilde{\mathcal A}_r$, i.e.,
\begin{align}\label{eq:lsubr}
L_r= p_\alpha^*g_r^{\alpha\beta} p_\beta;
\quad 
p_\alpha^*
=|\mathrm d r| (\det g_r)^{-1/2}p_\alpha |\mathrm d r|^{-1}(\det g_r)^{1/2}.
\end{align}
\end{lemma}

Denote the spherical part of the derivative $p$ by $p'$, or $p'=-\mathrm i\nabla'$, 
cf.\ \eqref{eq:176169}. 
The operator $p'$ is well-defined on $C^1(S_r)$ as well as on $C^1(M)$,
and we do not distinguish them.
It is clear from \eqref{eq:lsubr} that
for any $\xi,\zeta\in C^\infty_{\mathrm c}(S_r)$
\begin{align*}
\inp{\zeta,L_r\xi}_{\mathcal G_r}
=\int_{S_r} g_r^{ij}\overline{(p_i\zeta)}(p_j\xi)\,\mathrm d\tilde{\mathcal A}_r
=\langle p'\zeta,p'\xi\rangle_{\mathcal G_r}.
\end{align*} 
We can at this point 
use  local coordinates of $S$ to define and implement the integration, 
since in any case clearly the
radial derivative $\partial_r$ does not enter.
Hence in what follows we may consider $L_r$ as  
a self-adjoint operator on $\mathcal G_r$ 
defined by the Friedrichs extension of \eqref{eq:lsubr}
from $C^\infty_{\mathrm c} (S_r)\subseteq  \mathcal G_r$.
Then by an approximation argument it follows 
that for any $\phi \in \mathcal H^1$ the
restriction $\phi_{|S_r}\in
\mathcal D(L_r^{1/2})=\mathcal D (p')$ for almost every $r\geq r_0$.
In fact, we have for all $r\geq r_0$
\begin{align*}
\int_{r_0}^{r}
\|p'\phi_{|S_s}\|^2_{\mathcal G_s}\,\mathrm ds
=\int_{B_{r}\setminus B_{r_0}}
\ell^{ij}\overline{(p_i \phi)}(p_j\phi) (\det g)^{1/2}\,\mathrm d x
\leq \|\phi\|^2_{{\mathcal H}^1}.
\end{align*} 

The following formula will be useful when we compute 
and estimate the second spherical derivative $L$.

\begin{lemma}\label{lem:1761693}
For any $f\in C^2(M)$, if one abbreviates $\tilde y=\tilde y(t,{}\cdot{})$, then
\begin{align}
  \begin{split}
  L[f(\tilde y)]
  &
=
-\ell^{ij}
(\partial_i\tilde y^{\alpha})
(\partial_j\tilde y^{\beta})
(\nabla'{}^2 f)_{\alpha\beta} (\tilde y)
+ 
(L\tilde y)^\alpha(\partial _{\alpha} f)(\tilde y),
  \end{split}
\label{eq:17616945}
\end{align} 
where $\nabla'{}^2 f$ is defined in \eqref{eq:176169}, and 
\begin{align}
\begin{split}
L\tilde y^\alpha
&
=-\ell^{ij}\Bigl[
\partial_i\partial_j \tilde y^\alpha
-\Gamma_{ij}^k\partial_k \tilde y^\alpha
+\Gamma_{\beta\gamma}^\alpha
(\partial_i \tilde y^\beta)
(\partial_j \tilde y^\gamma)\Bigr]
\\&\phantom{{}={}}{}
-\tilde\eta\ell^{ij}(\nabla r)^\alpha
(\partial_i\tilde y^{\beta})
(\partial_j\tilde y^{\gamma}) (\nabla^2r)_{\beta\gamma}
\\&\phantom{{}={}}{}
+\Bigl[(\nabla^r\tilde\eta)(\nabla r)^j
+\tilde\eta(\Delta r)(\nabla r)^j
+\tfrac12\tilde\eta (\nabla|\mathrm dr|^2)^j\Bigr]\partial_ j\tilde y^{\alpha}.
\end{split}\label{eq:14.3.4.3.37AA}
\end{align}
Here the Roman and the Greek indices are those  
concerning $x$ and $\tilde y=\tilde y(t,x)$,
respectively,
differently from those in Lemma~\ref{lem:1762113}.
In addition, in the spherical coordinates in $E$ 
the first term on the right-hand side of \eqref{eq:17616945}
does not contain an $r$-derivative of $f$,
and neither does the second term.
\end{lemma}

Lemma~\ref{lem:1761693} motivates us to estimate $\partial_i\tilde y^\alpha$ and $L\tilde y$.
{We can estimate the former quantity 
through} the \textit{push-forward} $\ell_*(t,x)$
of $\ell(x)$ under the map $\tilde y(t,{}\cdot{})$, defined by 
$$\ell_*(t,x)=\Bigl(
\ell^{ij}(x)[\partial_i\tilde y^\alpha(t,x)][\partial_j\tilde y^\beta(t,x)]
\Bigr)_{\alpha,\beta}.
$$
 On the 
 the other hand, introducing  a ``backwards hitting time'' for $x\in E$ by 
\begin{align*}
  r^{\rm bht}(x)=\sup\bigl\{s \leq r(x)-r_0\,\big| \,
\tilde y(-s,x)\in M\big\},
\end{align*}   
for any $x\in E$ and $t\in (-r^{\rm bht}(x),0]$
the quantity $(L\tilde y^\alpha(t,x))_{\alpha=1,\dots,d}$ 
defines a tangent vector at $\tilde y(t,x)\in E$. 
It is in fact tangent to the $r$-sphere 
$S_{\tilde y(t,x)}=S_{r(x)+t}$ due to $L\tilde y^r=0$.


\begin{lemma}\label{lemma:L_10b}
Suppose Conditions~\ref{cond:12.6.2.21.13} and
\ref{cond:12.6.2.21.13a} (and $\sigma\in (0,\sigma')$).
Then 
for all $x\in E$ and $t\in (-r^{\rm bht}(x),0]$
\begin{align}
\begin{split}
\ell_*(t,x)
\le (d-1)\bigl[(r(x)+t)\big/r(x)\bigr]^{\sigma'}\ell(\widetilde y(t,x))
\end{split}\label{eq:13:9:34:16:20bb}
\end{align}
as quadratic forms on the fibers of the cotangent bundle. 
In spherical coordinates
the estimate 
\eqref{eq:13:9:34:16:20bb} reads:
For any $r>s\ge r_0$ and $\sigma\in S_s\subseteq S$
\begin{align}
\begin{split}
\ell(r,\sigma)
\le (d-1)(s/r)^{\sigma'}
\ell(s,\sigma).
\end{split}
\label{eq:16627}
\end{align}
If in addition \eqref{eq:smootrA} is fullfilled, 
then there exists $C>0$ such that
uniformly in $x\in E$ and $t\in (-r^{\rm bht}(x),0]$
\begin{align}
|L\tilde y(t,x)|
\le 
C\bigl[(r(x)+t)^{1/2}\big/r(x)\bigr]^{\min\{\sigma,\tau\}}.\label{eq:14.3.4.3.39bb}
\end{align} 
\end{lemma}

\subsection{Proof for general long-range model}
\label{subsec:proof 1}

In this subsection we show Theorem~\ref{thm:14.3.12.18.20}.

\begin{lemma}\label{lem:14.5.8.13.50}
Suppose Condition~\ref{cond:12.6.2.21.13bb}. Then 
for any $h\in C^2_{\mathrm c}(\mathcal I\times S)$ there exist the limits
\eqref{eq:53}.
\end{lemma}
\begin{proof} We shall employ the Cook--Kuroda method and Lemma
  \ref{lem:14.4.29.22.28}.  Due to time
  reversal invariance it suffices to consider the upper sign. 
  Let $h\in C^2_{\mathrm c}(\mathcal I\times S)$ be given. We divide
  the proof into {three steps. In the third step 
we treat the  two cases of Condition~\ref{cond:12.6.2.21.13bb}}.

\smallskip
\noindent 
\textit{Step I.}\quad We prepare for applying the Cook--Kuroda method
by introducing an energy cut-off $\bar \chi_n(H+C)$.
   Pick $C>0$ such that
  $H+C\geq 0$, and introduce $U^+_0(t)h$ by
  \eqref{eq:1609220}. 
Here we are going to show that 
\begin{align}
  \label{eq:energyL}
  \sup_{t\geq 1}\bigl\|\bar \chi_n(H+C)U^+_0(t)h\bigr\|_{\mathcal H}\to 0 
\text{ as }n\to \infty.
\end{align} 
By the Chebyshev type inequality we have 
\begin{align*}
  \bigl\|\bar \chi_n(H+C)U^+_0(t)h\bigr\|_{\mathcal H}^2\leq
  R_n^{-1}\inp{H+C}_{U^+_0(t)h},
\end{align*} 
and hence in order to prove \eqref{eq:energyL} it suffices to show that  
\begin{align}\label{eq:pest}
  \sup_{t\geq 1}\bigl\|pU^+_0(t)h\bigr\|_{\mathcal H}^2<\infty.
\end{align} 
To prove \eqref{eq:pest} let us write 
\eqref{eq:1609220} for short as
\begin{align}\label{eq:pest2}
\begin{split}
U^+_0(t)h(r,\sigma)
&=
\mathrm e^{Y(t,r,\sigma)}
\bigl(\partial_r\lambda_c(t,r,\sigma)\bigr)^{1/2}
Z(\lambda_c(t,r,\sigma),\sigma),
\end{split}\end{align}
where $Y\in C^k(\Omega_c)$ (according to $q_1\in C^k$)
and $Z\in C^2_{\mathrm c}((\lambda_1,\infty)\times D)$, 
$\overline D\subseteq S_{r_1}$, are defined by 
\begin{align*}
Y(t,r,\sigma)
&=\i K_1(t,r,\sigma)
-\tfrac12\int_{r_1}^{r}(\mathop{\mathrm{div}}\tilde\omega)(s,\sigma)\,\mathrm{d}s,\\
Z(\lambda,\sigma)
&= (2\pi)^{-1/2}\mathrm e^{-3\mathrm i\pi/4}
1_{\Omega_c(t)}(r,\sigma)
\mathrm e^{ \int_{r_0}^{r_1}[ \i \tilde
  b(s,\sigma)-(\mathop{\mathrm{div}}\tilde\omega)(s,\sigma)/2]\,\mathrm{d}s} 
h(\lambda,\sigma).
\end{align*} 
From this representation we obtain using \eqref{eq:Kder} and
\eqref{eq:derilamda} that
\begin{align*}
  \sup_{t\geq 1}\bigl\|p^rU^+_0(t)h\bigr\|_{\mathcal H}^2<\infty.
\end{align*} 
Therefore it suffices to show that 
\begin{align*}
  \sup_{t\geq 1}\bigl\|p'U^+_0(t)h\bigr\|_{\mathcal H}^2<\infty,
\end{align*} 
but here for later use we are going to show a sharper estimate
for any $\beta\in (0,\beta_c)$
\begin{align}\label{eq:basest}
\bigl\|p'U^+_0(t)h\bigr\|_{\mathcal H}\le C_\beta t^{-\beta}.
\end{align} 
We recall that $p'=-\mathrm i\nabla'$ is the spherical part of the derivative $p$, 
cf.\ \eqref{eq:176169}. 
Omitting variables, we can defferentiate \eqref{eq:pest2}: 
\begin{align}
\begin{split}
p'U^+_0h
&=
(p'Y)U^+_0h
+\tfrac12(p'\ln\partial_r\lambda_c)U^+_0h
+
\mathrm e^{Y}(\partial_r\lambda_c)^{1/2}
(p'\lambda_c)
(\partial_\lambda Z)_{|\lambda=\lambda_c}
\\&\phantom{{}={}}
+\mathrm e^{Y}(\partial_r\lambda_c)^{1/2}
(p'Z)_{|\lambda=\lambda_c}.
\end{split}
\label{eq:170721}
\end{align}
As for the first to third terms of \eqref{eq:170721},
if we note the isometry properties, cf.\ \eqref{eq:16092213}
and similar identity holding for 
$(\partial_r\lambda_c)^{1/2}(\partial_\lambda Z)_{|\lambda=\lambda_c}$
due to a change of variables,
it suffices to show that 
\begin{align}
  \sup_{t\geq 1,(t,r,\sigma)\in \Omega_c}
t^\beta\parb{|p'Y|+|p'\ln \partial_r\lambda_c|+|p'\lambda_c|}<\infty.
\label{eq:170724}
\end{align} 
 We compute using \eqref{eq:critical},  \eqref{eq:16092123} and \eqref{eq:1609211456}
 \begin{align*}
    p'Y(t,r,\sigma)
&=\biggl(\int^r_{r_1} 
p'\Bigl(\tilde b-\tfrac12\mathop{\mathrm{div}}\tilde\omega \Bigr)(s,\sigma)\,\d s
\biggr)_{
    |\lambda=\lambda_c(t,r,\sigma)},
\end{align*}
so that by \eqref{eq:16627} and \eqref{eq:twosidedest}  for $(t,r,\sigma)\in\Omega_c$
\begin{align}
\label{eq:est1}
   \begin{split}
t^{2\beta}|p'Y(t,r,\sigma)|^2
&=t^{2\beta}\ell^{ij}(r,\sigma)
\bigl[p_iY(t,r,\sigma)\bigr]
\bigl[p_jY(t,r,\sigma)\bigr]
\\&
\le 
C_1t^{2\beta}\biggl(\int_{r_1}^r 
(s/r)^{\sigma/2}
\Bigl|p'\Bigl(\pm\mathrm i\tilde b-\tfrac12\mathop{\mathrm{div}}\tilde \omega\Bigr)(s,\sigma)\Bigr|
\,\mathrm ds\biggr)_{
    |\lambda=\lambda_c(t,r,\sigma)}^2
\\&
    \le C_2t^{2\beta}r^{-2\beta}
    \le C_3.
    \end{split}
\end{align}
Similarly, since we have 
\begin{align}
\begin{split}
p'\lambda_c(t,r,\sigma)
&=-\parbb{\int^r_{r_1} b^{-2}(p'b)  (s,\sigma)\,\d s}_{|\lambda=\lambda_c(t,r,\sigma)}
\\&\phantom{{}={}}{}
\cdot\parbb{\int^r_{r_1} |\d r|^2b^{-3} (s,\sigma)\,\d s}^{-1}_{|\lambda=\lambda_c(t,r,\sigma)},
\end{split}
\label{eq:17072417}
\end{align} 
it follows that
\begin{align}\label{eq:est2}
  \sup_{(t,r,\sigma)\in \Omega_c}t^{\beta+1}|p'\lambda_c(t,r,\sigma)| \leq C_4.
 \end{align} 
We also compute using \eqref{eq:derilamda}
 \begin{align}
\begin{split}
   p'\ln \partial_r\lambda_c
&=-\Bigl [p'\parbb{b\int^r_{r_1} |\d
  r|^2 b^{-3} \,\d s}+(p'\lambda_c)\parbb{\partial_ \lambda \parbb{b\int^r_{r_1} |\d
  r|^2 b^{-3} \,\d s}}\Bigr]_{ |\lambda=\lambda_c}
\\ &\phantom{{}={}}{}\cdot \parbb{b\int^r_{r_1} |\d
  r|^2 b^{-3} \,\d s}^{-1}_{ |\lambda=\lambda_c},
\end{split}\label{eq:17072418}
 \end{align} so by estimating similarly again we obtain 
\begin{align}\label{eq:est3}
  \sup_{(t,r,\sigma)\in \Omega_c}
t^{\beta+1}|p'\ln \partial_r\lambda_c| \leq  C_5.
 \end{align} 
We have shown that the first to third terms of \eqref{eq:170721}
satisfy the desired estimate.
The fourth term of \eqref{eq:170721} can be bounded similarly using 
\eqref{eq:16627} and a change of variables.
Hence we obtain \eqref{eq:basest}, and in particular \eqref{eq:pest}.

\smallskip
\noindent
\textit{Step I\hspace{-.1em}I.}\quad
Here we reduce the proof of the lemma to the  following estimate
involving $L$ where $\chi=\chi_n(H+C)$:
\begin{align}
\int_{1}^\infty \bigl\|\chi L
U^+_0(t)h\bigr\|_{\mathcal H}\,\mathrm dt<\infty.
\label{eq:14.5.8.16.20ab}
\end{align}

Due  to  \eqref{eq:14.5.9.12.13BB} and \eqref{eq:energyL} 
 the lemma follows if we can show that for any large
$n \ge 1$
\begin{align*}
\int_{1}^\infty \bigl\|\chi\tfrac{\mathrm d}{\mathrm dt}\mathrm
e^{\mathrm itH}U^+_0(t)h\bigr\|_{\mathcal H}\,\mathrm dt<\infty.
\end{align*}  
Whence in turn it suffices to show that 
\begin{align}
\int_{1}^\infty \bigl\|\chi\parb{ H-G^+(t)}
U^+_0(t)h\bigr\|_{\mathcal H}\,\mathrm dt<\infty.
\label{eq:14.5.8.16.20}
\end{align} 
{By the expressions \eqref{eq:15091010b} and \eqref{eq:14.5.5.12.29B}
we can write 
\begin{align*}
  H-G^+(t)=
\tfrac 12(A-b_{\lambda_c})\tilde\eta (A-b_{\lambda_c})+\tfrac 12 L
+q_2+\tfrac14(\nabla^r\tilde\eta)(\Delta r).
\end{align*}
This leads to the formula
\begin{align*}
&  
\parb {H-G^+(t)}1_{\Omega_c(t)}(r,\sigma)\mathrm e^{Y(t,r,\sigma)}
=
1_{\Omega_c(t)}(r,\sigma)\mathrm e^{Y(t,r,\sigma)}
\parbb{\tfrac 12 p_r| \d r|^2p_r+\tfrac 12 L+q_2}.
\end{align*}
Whence, if we can verify \eqref{eq:14.5.8.16.20ab}
and 
\begin{align}\label{eq:14.5.8.16.20aa}
&\int_{1}^\infty \bigl\|\chi \mathrm e^{Y(t,\cdot)}
p_r| \d r|^2 p_r
(\partial_r\lambda_c(t,{}\cdot{}))^{1/2}
Z(\lambda_c(t,\cdot),\cdot)
\bigr\|_{\mathcal H}\,\mathrm
dt<\infty,
\end{align}
}then \eqref{eq:14.5.8.16.20} follows and the proof is done.

Now let us prove \eqref{eq:14.5.8.16.20aa}.
We use the product rule and
\eqref{eq:derilamda} to compute 
$$ p_r| \d r|^2 p_r
(\partial_r\lambda_c(t,{}\cdot{}))^{1/2}
{Z(\lambda_c(t,{}\cdot{}),{}\cdot{}).}$$ 
Only the  term 
\begin{align*}
  -|\d r|^2 \tfrac{\partial^2}{\partial
    r^2}(\partial_r\lambda_c(t,{}\cdot{}))^{1/2}
{Z(\lambda_c(t,{}\cdot{}),{}\cdot{})}
\end{align*} needs examination. In turn, when we expand it further, only a single
term might not contribute in agreement with
\eqref{eq:14.5.8.16.20aa}. This is a term that contains  a  second
order derivative of $q_{12}$. Explicitly it may be  expressed as 
\begin{align*}
  -\i\tfrac 12|\d r|^4 b^{-2}
    _{|\lambda=\lambda_c} {p_r\parb{\partial_r q_{12}}}
(\partial_r\lambda_c(t,{}\cdot{}))^{1/2}
{Z(\lambda_c(t,{}\cdot{}),{}\cdot{})}
\end{align*} that possibly do  not seem to agree with
\eqref{eq:14.5.8.16.20aa}. However since $\partial_r
q_{12}=O\parb{r^{-1-\rho/2}}$ we can pull the operator $p_r$ to the
left  in the corresponding time-integral, 
bound it with the factor $\chi$ (note that indeed $\chi p_r$ is
bounded) and then 
use that $t^{-1-\rho/2}$ is integrable.
Hence \eqref{eq:14.5.8.16.20aa} is verified,
and the proof of the lemma reduces to \eqref{eq:14.5.8.16.20ab}.

\smallskip
\noindent
\textit{Step I\hspace{-.1em}I\hspace{-.1em}I.}\quad 
Finally we
verify the estimate \eqref{eq:14.5.8.16.20ab} 
by using 
Condition~\ref{cond:12.6.2.21.13bb}.

First we assume (\ref{item:14.5.1.8.30}) of Condition~\ref{cond:12.6.2.21.13bb}.
To prove the estimate \eqref{eq:14.5.8.16.20ab} 
it suffices to show
\begin{align*}
  \int_{1}^\infty \bigl\|p'
U^+_0(t)h\bigr\|_{\mathcal H}\,\mathrm dt<\infty.
\end{align*} 
However, this bound is easily verified 
by noting that \eqref{eq:basest} is valid for  some
$\beta>1$ in this case.

Hence for the rest of the proof we 
assume (\ref{item:14.5.1.8.31}) of Condition~\ref{cond:12.6.2.21.13bb}. 
Using the fibration \eqref{eq:15091010} of $L$ in spherical coordinates
and the notation as in \eqref{eq:pest2},
we can write 
\begin{align}\label{eq:primeder}
\begin{split}
L_rU^+_0h&=
(L_rY)U^+_0h
+\tfrac12(L_r\ln\partial_r\lambda_c)U^+_0h
+\mathrm e^{Y}(\partial_r\lambda_c)^{1/2}
(L_r\lambda_c)(\partial_\lambda Z)_{|\lambda=\lambda_c}
\\&\phantom{{}={}}
+\mathrm e^{Y}(\partial_r\lambda_c)^{1/2}(L_rZ)_{|\lambda=\lambda_c}
+|p'Y|^2U^+_0h
+\ell^{ij}(p_iY)(p_j\ln\partial_r\lambda_c)U^+_0h
\\&\phantom{{}={}}
+2\ell^{ij}(p_iY)(p_j\lambda_c)\mathrm e^{Y}(\partial_r\lambda_c)^{1/2}(\partial_\lambda Z)_{|\lambda=\lambda_c}
\\&\phantom{{}={}}
+2\ell^{ij}(p_iY)\mathrm e^{Y}(\partial_r\lambda_c)^{1/2}(p_jZ)_{|\lambda=\lambda_c}
+\tfrac14|p'\ln\partial_r\lambda_c|^2U^+_0h
\\&\phantom{{}={}}
+\ell^{ij}(p_i\ln\partial_r\lambda_c)\mathrm e^{Y}(\partial_r\lambda_c)^{1/2}(p_j\lambda_c)(\partial_\lambda Z)_{|\lambda=\lambda_c}
\\&\phantom{{}={}}
+\ell^{ij}(p_i\ln\partial_r\lambda_c)\mathrm e^{Y}(\partial_r\lambda_c)^{1/2}(p_jZ)_{|\lambda=\lambda_c}
\\&\phantom{{}={}}
+\mathrm e^{Y}(\partial_r\lambda_c)^{1/2}|p'\lambda_c|^2(\partial_\lambda^2 Z)_{|\lambda=\lambda_c}
\\&\phantom{{}={}}
+2\ell^{ij}\mathrm e^{Y}(\partial_r\lambda_c)^{1/2}(p_i\lambda_c)(\partial_\lambda p_j Z)_{|\lambda=\lambda_c}
,
\end{split}
\end{align} 
cf.\ \eqref{eq:170721}. 
Similarly to Step I\hspace{-.1em}I, we can show that 
the fifth to thirteenth terms of \eqref{eq:primeder} are 
$O(t^{-2\beta})$ for any $\beta\in (0,\beta_c)$.
Here we in particular implemented \eqref{eq:170724} and the isometric properties 
due to a change of variables.
To bound the first term of \eqref{eq:primeder}
we first compute by \eqref{eq:critical},  \eqref{eq:16092123}, \eqref{eq:1609211456}
and \eqref{eq:17616945} that 
 \begin{align*}
L_rY(t,r,\sigma)
&=\biggl(L_r\int^r_{r_1} 
\Bigl(\tilde b-\tfrac12\mathop{\mathrm{div}}\tilde\omega \Bigr)(s,\sigma)\,\d s
\biggr)_{|\lambda=\lambda_c(t,r,\sigma)}
\\&
=
\biggl(\ell^{ij}(r,\sigma)
\int^r_{r_1} 
\Bigl[\nabla'^2\bigl(\tilde b-\tfrac12\mathop{\mathrm{div}}\tilde\omega \bigr)\Bigr]_{ij}(s,\sigma)\,\d s
\biggr)_{|\lambda=\lambda_c(t,r,\sigma)}
\\&\phantom{{}={}}{}
+\biggl(\int^r_{r_1}
(L\tilde y)^i (s-r,r,\sigma)
\Bigl[\partial_i\bigl(\tilde b-\tfrac12\mathop{\mathrm{div}}\tilde\omega \bigr)\Bigr](s,\sigma)\,\d s
\biggr)_{|\lambda=\lambda_c(t,r,\sigma)}.
\end{align*}
Then by \eqref{eq:16627}, \eqref{eq:14.3.4.3.39bb}
and (\ref{item:14.5.1.8.31}) of Condition~\ref{cond:12.6.2.21.13bb}
we obtain that the first term of \eqref{eq:primeder} is
$O(t^{-2\beta})$ for any $\beta\in (0,\beta_c)$.
The second to fourth terms of \eqref{eq:primeder} are treated basically
in the same manner as the first.
We first use the formula \eqref{eq:17616945} with $t=0$ 
to reduce the estimate to those of 
$$\nabla'\ln\partial_r\lambda_c,
\quad 
\nabla'^2\ln\partial_r\lambda_c,
\quad  
\nabla'\lambda_c,\quad 
\nabla'^2\lambda_c,\quad 
(\nabla'Z)_{|\lambda=\lambda_c},\quad
(\nabla'^2Z)_{|\lambda=\lambda_c}.$$ 
The first order derivatives are already estimated in Step I\hspace{-.1em}I.
The second order derivatives are computed, e.g., from \eqref{eq:17072417} and \eqref{eq:17072418},
and then estimated similarly.
Hence we can conclude that $L_rU^+_0h$ is $O(t^{-2\beta})$ for any $\beta\in(0,\beta_c)$,
and this in particular implies the integrability \eqref{eq:14.5.8.16.20ab}.
\end{proof}

Next  we extend the wave operator  to  a bounded operator $W^\pm\colon \widetilde{\mathcal H}_{\mathcal I}
\to \mathcal H_{\mathcal I}$ 
and show \eqref{eq:54}.
\begin{lemma}\label{lem:14.5.8.13.51}
Suppose  Condition~\ref{cond:12.6.2.21.13c}  and that there exist the limits \eqref{eq:53} for any 
$h\in C^2_{\mathrm c}(\mathcal I\times S)$.
Then $W^\pm$ extend to  isometric   operators $W^\pm\colon \widetilde{\mathcal H}_{\mathcal I}
\to \mathcal H_{\mathcal I}$,
and \eqref{eq:54} holds. In particular $W^\pm\colon \widetilde{\mathcal H}_{\mathcal I}
\to \mathcal H_{\mathcal I}$ are unitary. 
\end{lemma}
\begin{proof}
{Again only the upper sign is considered.} Due to \eqref{eq:16092213} and
\eqref{eq:14.5.9.12.13BB} the operator $W^+$ is isometric. The
property \eqref{eq:54} (to be shown) implies the unitarity. Whence it
suffices to show that $W^+$ maps into $\mathcal H_{\mathcal I}$ and
that \eqref{eq:54} is fulfilled.

We shall  proceed  partially following \cite[Appendix A]{HS}.
For any $h\in C^2_{\mathrm c}(\mathcal I\times S)$ and $\psi\in
{\mathcal H_{1+}}\cap \vH^1$
\begin{align*}
2\pi\mathrm i\inp{\psi, W^+ h}
&=2\pi\mathrm i\lim_{\epsilon\downarrow 0}\epsilon \int_0^\infty\mathrm e^{-\epsilon t}
\bigl\langle \psi,\mathrm e^{\mathrm itH}\bar \chi_nU^+(t)h\bigr\rangle\,\mathrm dt\\
&=\lim_{\epsilon\downarrow 0}\epsilon\int_0^\infty\Bigl\langle \mathrm
e^{\mathrm -\i t(H-\mathrm i\epsilon )}\psi,
\int_{\mathcal I} 
\mathrm e^{-\mathrm it\lambda}\bar \chi_n\phi_\lambda^+[h(\lambda)]\,\mathrm d\lambda
\Bigr\rangle \mathrm dt;
\end{align*} here the cut-off function  $\bar \chi_n $ is 
chosen  with  $n$ big enough (possibly depending  on  $h$) 
 to assure
 the property  $\bar \chi_n \phi_\lambda^+[h(\lambda)]\in\vN$ for all
 $\lambda$. Moreover, here
 and below,  we use the $L^2$-pairing of the spaces  $\mathcal
H_s$ and $
\mathcal H_{-s}$ for any real $s$ (denoted by
$\langle{}\cdot{},{}\cdot{}\rangle$ as   the inner product).   
Note that by  commutation we may bound
\begin{align*}
\int_0^\infty
\bigl\|r\mathrm e^{\mathrm -\i t(H-\mathrm i\epsilon )}\psi\bigr\|
\,\mathrm dt<\infty,
\end{align*}
so that we may insert a configuration cut-off and obtain
\begin{align*}
2\pi\mathrm i\inp{\psi, W^+ h}
=\lim_{\epsilon\downarrow 0}\lim_{m\to \infty}\epsilon\Bigl\langle \psi,
\int_0^\infty\mathrm e^{\mathrm it(H+\mathrm i\epsilon)}\int_{\mathcal I} 
\mathrm e^{-\mathrm it\lambda} \chi_m \bar \chi_n \phi^+_\lambda[h(\lambda)]\,\mathrm d\lambda
\mathrm dt
\Bigr\rangle.
\end{align*} 
Now we can change the order of integration  and then compute, abbreviating $\phi_\epsilon= R(\lambda+\mathrm
i\epsilon)\psi$,
\begin{align*}
\begin{split}
&2\pi\mathrm i\inp{\psi, W^+ h}\\
&=\lim_{\epsilon\downarrow 0}\lim_{m\to \infty}
\int_{\mathcal I} \epsilon\Bigl\langle \psi,
\int_0^\infty\mathrm e^{\mathrm it(H-\lambda+\mathrm i\epsilon )}\chi_{n,m}\phi^+_\lambda[h(\lambda)]\,
\mathrm dt
\Bigr\rangle\,\mathrm d\lambda\\
&=\lim_{\epsilon\downarrow 0}\lim_{m\to \infty}
\int_{\mathcal I} \bigl\langle \psi,
\mathrm i\epsilon R(\lambda-\mathrm i\epsilon)\chi_{n,m} \phi^+_\lambda[h(\lambda)]
\bigr\rangle\,\mathrm d\lambda\\
&=\lim_{\epsilon\downarrow 0}\lim_{m\to \infty}
\int_{\mathcal I} 
\parbb{\bigl\langle \psi,\chi_{n,m} \phi^+_\lambda[h(\lambda)]\bigr\rangle
-\bigl\langle \psi,R(\lambda-\mathrm i\epsilon)(H-\lambda)\chi_{n,m} \phi^+_\lambda[h(\lambda)]\bigr\rangle}
\,\mathrm d\lambda\\
&=\int_{\mathcal I} 
\bigl\langle \psi,\bar \chi_n \phi^+_\lambda[h(\lambda)]\bigr\rangle
\,\mathrm d\lambda
-
\lim_{\epsilon\downarrow 0}\lim_{m\to \infty}
\int_{\mathcal I} 
\bigl\langle \phi_\epsilon,
 (H-\lambda)\chi_{n,m} \phi^+_\lambda[h(\lambda)]\bigr\rangle
\,\mathrm d\lambda.
\end{split}
\end{align*} We shall  compute the double limit by an integration by
parts procedure  first  rewriting  the integrand by
using the spherical decomposition formula \cite[(2.8a)]{IS3} (as in  the proof of \cite[Theorem
1.14]{IS3}). A commutation error disappears when taking  the
$m$-limit and whence,  more precisely,
\begin{align*}
  \lim_{m\to \infty}
2\int_{\mathcal I} 
\bigl\langle \phi_\epsilon,
(H-\lambda)\chi_{n,m} \phi^+_\lambda[h(\lambda)]\bigr\rangle
\,\mathrm d\lambda
=\lim_{m\to \infty}\int_{\mathcal I} (T_1+T_2+T_3)\,\mathrm d\lambda,
\end{align*} where (for some $\kappa>1$ and $\i p'$ being the
covariant derivative on $S_r$)
\begin{align*}
  T_1&=\inp{(A+\bar a)\phi_\epsilon, \chi_m\tilde\eta(A-a)\bar \chi_n\phi^+_\lambda[h(\lambda)]},\\
T_2 &=\int^\infty _{r_0} \chi_m(r)\inp{p'\phi_\epsilon, p'\bar \chi_n\phi^+_\lambda[h(\lambda)]}_{\vG_r}\,\d r,\\
T_3 &=\int^\infty _{r_0} \chi_m(r)\inp{\phi_\epsilon,
    O(r^{-\kappa})\bar \chi_n\phi^+_\lambda[h(\lambda)]}_{\vG_r}\,\d r.
\end{align*} We can compute the $m$-limit from this representation,
however since we need to compute a double limit it turns out to be
better to compute the $\epsilon$-limit first. This change of order is
legitimate since the $m$-limit exists \emph{uniformly} in small
$\epsilon$ due to Theorem \ref{prop:radiation-conditions} and
\eqref{eq:13:9:34:16:20bb} (at this point
Condition~\ref{cond:12.6.2.21.13c} is crucial). By taking the
$\epsilon$-limit inside the $\lambda$-integrals we can then compute
the $m$-limit for each term after picking up (for $T_1$) the factor
$\chi_m'$ (i.e. the $r$-derivative of $\chi_m$) from a commutation
(precisely as in the proof of \cite[Theorem 1.14]{IS3}). Now the
corresponding term \emph{does} contribute in the $m$-limit and we
obtain the following formula, abbreviating $\phi= R(\lambda+\mathrm
i0)\psi$,
\begin{align*}
  -\lim_{\epsilon\downarrow 0}\lim_{m\to \infty}&
\int_{\mathcal I} 
\bigl\langle \phi_\epsilon,
 (H-\lambda)\chi_{n,m} \phi^+_\lambda[h(\lambda)]\bigr\rangle
\,\mathrm d\lambda\\
&=\int_{\mathcal I} \parbb{\i\inp {F^+(\lambda)\psi, h(\lambda)}_{\vG} -\inp{\psi,
    \bar \chi_n \phi_\lambda^+[h(\lambda)]}}\,\mathrm d\lambda.
\end{align*} By combining this with the previous computation we observe a
cancellation and conclude that
\begin{align}\label{eq:idenbas}
  2\pi\mathrm i\inp{\psi, W^+ h}=\int_{\mathcal I} \i\inp
  {F^+(\lambda)\psi, h(\lambda)}_{\vG}\,\mathrm d\lambda.
\end{align} 

Let us consider $\psi=f(H)\breve \psi$ where $f\in
C_\c^\infty(\R)$ and $\breve\psi\in {\mathcal H_{1+}}\cap \vH^1$. We
obtain from \eqref{eq:idenbas} that
\begin{align*}
  2\pi\mathrm i\inp{(I-f(H))\breve\psi, W^+ h}=\int_{\mathcal I} \i(1-f(\lambda))\inp {F^+(\lambda)\breve\psi, h(\lambda)}_{\vG}\,\mathrm d\lambda. 
\end{align*}
  Clearly if $f$ is chosen
such that $f(\lambda)h(\lambda,\cdot)=h(\lambda,\cdot)$ for all
$\lambda$ the right-hand
side vanishes and whence by density we conclude that $W^+$ maps into
$\vH_{\vI}$ and extends to an isometry $W^+:\tilde \vH_\vI \to
\vH_{\vI}$. Moreover by choosing $f\in C_\c^\infty(\vI)$ and 
$\breve\psi\in {\mathcal H_{1+}} $ and then using $\psi=f(H)\breve
\psi$ in 
\eqref{eq:idenbas} it follows (again by density) that
$(W^+)^*\psi=F^+\psi$ for all $\psi\in\vH_{\vI}$ showing   
\eqref{eq:54}.
\end{proof}

Theorem \ref{thm:14.3.12.18.20}  
is a direct consequence of 
Lemmas~\ref{lem:14.5.8.13.50} and \ref{lem:14.5.8.13.51}.
 In fact  using only Lemma
 \ref{lem:14.5.8.13.51}  we obtain   a version of Theorem \ref{thm:dist-four-transf} 
 without need for  Condition~\ref{cond:12.6.2.21.13bb}.
\begin{corollary}\label{cor:14.5.10.7.30} Suppose  Condition~\ref{cond:12.6.2.21.13c}  and that there exist the limits \eqref{eq:53} for any 
  $h\in C^2_{\mathrm c}(\mathcal I\times S)$.
Then the operators  $F^\pm\colon \mathcal H_{\mathcal
  I}\to\widetilde{\mathcal H}_{\mathcal I}$ are  unitary diagonalizing transformations.
\end{corollary}
\begin{proof}
We know that $F^+$ is   an isometry,
and hence it suffices to show that $F^+$ is onto.
But  since $W^+$ is an isometry 
the conclusion follows from  \eqref{eq:54}.
\end{proof}

\subsection{Proofs for short-range and Dollard type models}\label{subsec:proof 3}
In this subsection we prove Theorems~\ref{thm:14.3.12.18.20sr} and
\ref{thm:14.3.12.18.20srb} stated in Subsection~\ref{subsec:170531}
under the short-range assumption 
and Theorem~\ref{thm:14.3.12.18.20dol} stated in 
Subsection~\ref{subsec:17053116} under the Dollard type assumption.

Let us begin with the proof of 
Theorem~\ref{thm:14.3.12.18.20dol}  
since the  short-range model can be treated easily using a bound from
the proof. It suffices to do the upper case due to time reversal 
invariance.

\begin{proof}[Proof of Theorem~\ref{thm:14.3.12.18.20dol}]
Let $\theta=\theta(\sigma, \lambda)$ be the real-valued 
function appearing in Corollary~\ref{cor:dollard}. Fix  any $h\in C^2_{\mathrm c}(\mathcal I\times
S)$ and corresponding quantities $\lambda_1$, $D$ and $r_1$ (determining
$ U^+_0(t)h$ by \eqref{eq:1609220}). 
In order to prove the existence of the limit \eqref{eq:53dol} for this
$h$ it
suffices to show that 
\begin{align}\label{eq:asysim}
  \|U_{\rm do}^+(t)\check h- U^+_0(t)h\|\to 0
\text{ as }t\to \infty;\quad  h:=\e^{\i \theta}\check h,
\end{align}
cf.\ Theorem~\ref{thm:14.3.12.18.20} and Lemma~\ref{lem:14.4.29.22.28}.  By the properties  of $h$ we can find $\lambda_2>\lambda_1$ such that
 $\lambda\leq \lambda_2$ on $\supp h$. Then,  recalling 
 the quantities of \eqref{eq:1609220},   we let 
 \begin{align*}
   \Omega^2_c(t)=\{(r,t)\in\Omega_c(t)| \quad \lambda_c\leq
 \lambda_2\}.
 \end{align*} 
 First  we note 
\begin{align*}
\|U^+_0(t)h\|=\|h\|\geq \|U_{\rm do}^+(t)\check h\|,
\end{align*}
 reducing  the proof of \eqref{eq:asysim} to  showing
\begin{align}\label{eq:asysimb}
  \|1_{\Omega^2_c(t)}U_{\rm do}^+(t)\check h- U^+_0(t)h\|\to 0
\text{ as }t\to  \infty.
\end{align}

Hence let us prove \eqref{eq:asysimb}.
Obviously we need to `compare' the expressions  \eqref{eq:1609220} and \eqref{17053120}
with the factor $\e^{\i \theta}$.  By using 
\eqref{eq:dollard2} and \eqref{eq:twosidedest} 
 we obtain that  
\begin{align*}
  \sup_{(r,\sigma)\in \Omega^2_c(t)}\int_{{r}}^\infty\big (
b_{\lambda_c}(s,\sigma)-b_{{\rm do,\lambda_c}} (s,\sigma) \big )\,\mathrm{d}s
=o(t^0).
\end{align*} 
 Whence it  suffices for  \eqref{eq:asysimb} to show
 the following asymptotics 
uniformly in 
$(r,\sigma)\in \Omega^2_c(t)$:
\begin{align*}
\lambda_c&=\lambda_0(\sigma)+\tfrac 12 \tfrac {(r-r_0)^2}{t^2}+o(t^0),\\
\partial_r\lambda_c&=(r-r_0)/t^2+o(t^{-1}),\\
      K_{\rm do}&=\int_{{r_0}}^r
b_{{\rm do},\lambda_c}(
s,\sigma)\,\mathrm{d}s-\lambda_c t+o(t^0).
\end{align*}  
 Now to  prove these asymptotics we  do  a Taylor expansion of  the
 integrand of \eqref{eq:1609211456} 
 to  find a continuous function $f_1(\lambda,\sigma)$ ($\lambda$ being
 close to $\lambda_c$) 
such that
\begin{align*}
  \partial_\lambda\Theta_1(\lambda,t,r,\sigma)
=b_{\rm sr}^{-1}(r-r_1) +\tfrac 12 b_{\rm sr}^{-3}\int_{r_1}^r Q\,\d
  s+f_1 +O(t^{-\epsilon})-t,
\end{align*}
where $\epsilon\in(0,1)$ is a constant 
satisfying \eqref{eq:dollard} and 
\begin{align*}
Q
=Q(x)=2(q^{\rm ex}_1-\lambda_0(\sigma)).
\end{align*}
Alternatively, we can rewrite 
for some  continuous function $f_2(\lambda,\sigma)$
\begin{align*}
  \partial_\lambda\Theta_1(\lambda,t,r,\sigma)
=b_{\rm sr}^{-1}(r-r_0) 
+\tfrac 12 b_{\rm sr}^{-3}\int_{r_0}^r Q\,\d
  s+f_2 +O(t^{-\epsilon})-t.
\end{align*} 
Let us substitute $\lambda=\lambda_c(t,r,\sigma)$
and abbreviate
$f_2=f_2(\lambda_c,\sigma)$
and $b_{\rm sr}=b_{\rm sr}(\lambda_c,\sigma)$.
This leads to
\begin{align}\label{eq:1asy}
  b_{\rm sr}=\tfrac{r-r_0}t+\tfrac 12 \tfrac t{(r-r_0)^2}\int_{r_0}^r Q\,\d
  s+\tfrac {(r-r_0)} {t^2}f_2 +O(t^{-1-\epsilon}),
\end{align} and consequently 
\begin{align}\label{eq:2asy}
  \lambda_c=\lambda_0(\sigma)+\tfrac 12 \tfrac {(r-r_0)^2}{t^2}+\tfrac 12 \tfrac 1{(r-r_0)}\int_{r_0}^r Q\,\d
  s+\tfrac {(r-r_0)^2} {t^3}f_2  +O(t^{-1-\epsilon}). 
\end{align} 
In particular 
\begin{align*}
  \lambda_c=\lambda_0(\sigma)+\tfrac 12 \tfrac {(r-r_0)^2}{t^2}+O(t^{-(1+\epsilon)/2}). 
\end{align*} In turn using
\begin{align*}
    \partial_r\lambda_c=
    \biggl(b_{\lambda_c}(r,\sigma)\int_{{r_1}}^r b_{\lambda_c}(s,\sigma)^{-3}\,\mathrm{d}s\biggr)^{-1},
  \end{align*} we obtain 
\begin{align*}
    \partial_r\lambda_c= (r-r_0)/t^2+O(t^{-(3+\epsilon)/2}).
  \end{align*} Finally we compute using \eqref{eq:1asy} and \eqref{eq:2asy}
  \begin{align*}
    &K_{\rm do}-\int_{{r_0}}^r
b_{{\rm do},\lambda_c}(
s,\sigma)\,\mathrm{d}s+\lambda_c t\\
&=K_{\rm do}-(r-r_0)
\biggl((r-r_0)/t+\tfrac 12 \tfrac t{(r-r_0)^2}\int_{r_0}^r Q\,\d
  s+\tfrac {(r-r_0)} {t^2}f_2 \biggr)+ \tfrac 12 \tfrac t{(r-r_0)}\int_{r_0}^r Q\,\d
  s\\&\phantom{{}={}}{}
+\biggl(\lambda_0(\sigma)+\tfrac 12 \tfrac {(r-r_0)^2}{t^2}+\tfrac 12 \tfrac 1{(r-r_0)}\int_{r_0}^r Q\,\d
  s+\tfrac {(r-r_0)^2} {t^3}f_2 \biggr)t+O(t^{-\epsilon})\\
&=O(t^{-\epsilon}).
  \end{align*} 
Whence we have shown that \eqref{eq:asysimb} holds.

The rest of the assertions is clear from 
\eqref{eq:53dol}, \eqref{eq:asysim},
Theorem~\ref{thm:14.3.12.18.20} and 
Corollary~\ref{cor:dollard}.
\end{proof}

\begin{proof}[Proof of Theorem~\ref{thm:14.3.12.18.20sr}]
{As we already noted $q$ is of
Dollard type since it is of short-range type. }
Let $\theta=\theta(\sigma, \lambda)$ be the real-valued 
function appearing in Corollary~\ref{cor:strongbb}. 
Similarly to the proof of Theorem~\ref{thm:14.3.12.18.20dol},
in order to prove the existence of the limit \eqref{eq:53sr} it
suffices to show that for all $h\in C^2_{\mathrm c}(\mathcal I\times
S)$ 
\begin{align}\label{eq:asysimc}
  \|U_{\rm sr}^+(t)\check h- U^+_0(t)h\|\to 0
\text{ as }t\to  \infty;\quad  h:=\e^{\i \theta}\check h.
\end{align}
If we compare the expressions \eqref{eq:16091819} and \eqref{17053120},
then we can see that \eqref{eq:asysimc} follows from \eqref{eq:asysim}.
Here we note that the functions $\theta$ in 
\eqref{eq:asysim} and \eqref{eq:asysimc} are 
chosen differently as in  Corollary~\ref{cor:dollard} and 
Corollary~\ref{cor:strongbb}, respectively.
Hence we are  done with \eqref{eq:53sr}.

The rest of the assertions is clear from 
\eqref{eq:53sr}, \eqref{eq:asysimc},
Theorem~\ref{thm:14.3.12.18.20} and 
Corollary~\ref{cor:strongbb}.
\end{proof}

We above proof is short since we could use the proof of Theorem
\ref{thm:14.3.12.18.20dol}, more precisely \eqref{eq:asysim}. A
different and 
possibly more appealing procedure would be show \eqref{eq:asysimc}
directly 
by 
mimicking  the proof of Theorem \ref{thm:14.3.12.18.20dol} for  the
short-range setting.

\begin{proof}[Proof of Theorem~\ref{thm:14.3.12.18.20srb}]
Under Condition~\ref{cond:14.5.24.7.4} the existence of the limits
\eqref{eq:53srb} follows by \cite{IS1}.
Then by Lemma~\ref{lem:170601} and \eqref{eq:asysimc} 
the assumptions of Lemma~\ref{lem:14.5.8.13.51} and 
Corollary~\ref{cor:14.5.10.7.30} are fulfilled.
Now we can argue as at the end of the proof of of
  Theorem~\ref{thm:14.3.12.18.20sr} and complete the proof of the  theorem.
\end{proof}

\end{document}